\documentclass{aptpub}

\numberwithin{equation}{section}
\numberwithin{theorem}{section}
\numberwithin{lemma}{section}
\numberwithin{corollary}{section}
\numberwithin{definition}{section}
\numberwithin{remark}{section}

\usepackage{graphicx}
\usepackage{mathrsfs}
\usepackage{xcolor}
\usepackage{epstopdf}

\authornames{Thoppe, Yogeshwaran, Adler} % insert the authors here for use in running head
\shorttitle{Topology of dynamic clique complexes} % insert short title here for use in running head

\newcommand{\ER}{Erd{\H o}s-R{\' e}nyi }
\renewcommand{\Pr}{\mathbb{P}}
\newcommand{\CliqueBar}{$\bar{f}_{n, k}$}
\newcommand{\EulerBar}{$\bar{\chi}_{n}$}
\newcommand{\BetaBar}{$\bar{\beta}_{n, k}$}
\newcommand{\ver}[1]{$\text{ver}(\bar{#1})$}
\newcommand{\pair}[1]{$\text{pair}(\bar{#1})$}

\newcommand{\cX}{{\mathcal X}}
\newcommand{\cY}{{\mathcal Y}}

\newcommand{\cK}{{\mathcal K}}

\newcommand{\beq}{\begin{eqnarray}}
\newcommand{\eeq}{\end{eqnarray}}
\newcommand{\beqq}{\begin{eqnarray*}}
\newcommand{\eeqq}{\end{eqnarray*}}

\def\real{\mathbb{R}}

\def\:{:\,}

\allowdisplaybreaks[2]
\begin{document}

\title{On the evolution of topology in dynamic clique complexes} % insert title - use \\ if it requires more than one line.

\authorone[Technion - Israel Institute of Technology]{Gugan C. Thoppe} % Affiliation is just the name of your university or institution
\authortwo[Indian Statistical Institute]{D. Yogeshwaran}
\authorthree[Technion  - Israel Institute of Technology]{Robert J. Adler}

\addressone{Faculty of Electrical Engineering, Technion, Haifa, ISRAEL, 32000.} % Your postal address goes here.
\addresstwo{Statistics and Mathematics Unit, ISI, Bangalore, INDIA, 560059.}
\addressthree{Faculty of Electrical Engineering, Technion, Haifa, ISRAEL, 32000.}

\begin{abstract}
We consider a time varying analogue of the \ER graph and study the topological variations of its associated clique complex. The dynamics of the graph are stationary and are determined by the edges, which  evolve independently as continuous time Markov chains. Our main result is that when the edge  inclusion probability is of the form $p = n^\alpha$, where $n$ is the number of vertices and $\alpha \in (-1/k, -1/(k + 1)),$ then the process of the normalized $k-$th Betti number of these dynamic clique complexes converges  weakly to the Ornstein-Uhlenbeck process as $n \to \infty.$
\end{abstract}

 % insert keywords separated by a semicolon


\section{Introduction}
\label{sec:Introduction}

The classic  \ER graph $G(n, p)$ is well known  as the random graph on $n$ vertices where each edge appears with probability $p$, independently of the others. It is ubiquitous in applied literatures dealing with network models and, despite its apparent simplicity, has been of theoretical interest ever since Erd\"os and R\'enyi, over half a century ago in \cite{erdos1959random},  established a sharp threshold for its connectivity. They showed that,  for fixed  $\epsilon > 0$ ,
as $n\to\infty$,
%\[
%\Pr\{G(n,p) \text{ is connected}\} \rightarrow 1 \text{ as } n \rightarrow \infty,
%\]
%while if $p \leq (1 - \epsilon) \log(n)/n$, then
%\[
%\Pr\{G(n,p) \text{ is disconnected}\} \rightarrow 1 \text{ as } n \rightarrow \infty.
%\]

\beqq
\Pr\{G(n,p) \text{ is connected}\} \ \rightarrow \
\begin{cases}
1  &  \text{if $p \geq (1 + \epsilon) \log(n)/n$}, \\
0  &  \text{if $p \leq (1 - \epsilon) \log(n)/n$}.
 \end{cases}
\eeqq

Allowing for the interpretation that connectedness is a (almost trivial) topological property, their result can be considered as the first result describing a topological phase
transition in a random graph. Since 1959, a substantial literature has grown around the properties of the \ER graph, providing much finer detail than the original result.
A more recent literature, some of which we shall describe briefly below, has considered more detailed topological information about  objects generated by $G(n,p)$.

In this paper, we take all of this a step further, applying these richer probabilistic results in the topological setting,  to  temporally evolving \ER graphs.
We need a few definitions, or at least descriptions, in order to define what we mean by this.

\subsection{Some background}
\label{subsec:tb}

\subsubsection{Dynamic \ER graphs:}
\label{subsubsec:ER}
The dynamic \ER graph depends on three parameters: the number of nodes, $n\in \mathbb{N}$, the connectivity probability $p\in [0,1]$, and a rate, $\lambda > 0$.
Denoted by $\{G(n, p, t): t \geq 0\}$, it is a
time-varying subgraph of the complete graph  on $n$ vertices with the following properties.
%\begin{itemize}
%\item
\newline
(i) The initial value  $G(n, p, 0)$ is distributed as the (static) \ER graph $G(n, p)$.
\newline
%\item
(ii) For $t \geq 0$, each edge independently evolves as a continuous time on/off Markov chain. The waiting time in the states `off' and `on' are exponential with parameters $\lambda p$ and $\lambda (1-p)$ respectively.
%\end{itemize}

If $e(t)$ denotes  the state of one of these edges at time $t$, then it follows immediately from the above description that, for any $t_1, t_2$,
\beq
\label{eqn:EdgeTransition11}
\Pr\{e(t_2) =\text{on}\, \big|\, e(t_1)=\text{on}\} &=&  p + (1 - p)e^{-\lambda |t_2 - t_1|},
\eeq
and
\beq
\label{eqn:EdgeTransition00}
\Pr\{e(t_2) = \text{off}\,\big| \, e(t_1) = \text{off}\} &=&  (1 - p) + pe^{-\lambda |t_2 - t_1|}.
\eeq
From this it follows that, for any $t \geq 0$,
\begin{equation}
\label{eqn:EdgeProbability}
\Pr\{e(t) = \text{on}\} = p.
\end{equation}
Consequently, $\{G(n, p, t): t \geq 0\}$ is a stationary reversible Markov process and, for each $t \geq 0$, it is a realisation of the (static) \ER graph, $G(n,p)$.

The dynamic \ER graph described here is an example of a continuous time `Edge Markovian Evolving Graph' (EMEG), a  class of dynamic models that has often been used to model real world dynamic networks. In particular, if one thinks of the static \ER graph as a simple, but generic model for `faulty connections' between nodes, then the dynamic version is clearly relevant to   `Intermittently Connected Mobile Networks' (ICMNs) \cite{musolesi2005adaptive,spyropoulos2005spray}. The ICMNs have given rise to  many  interesting new questions, such as temporal connectivity \cite{basu2010modeling,clementi2010flooding} and dynamic community detection \cite{clementi2014distributed}, all related, in one way or another, to issues of connectivity. For us, however, the importance of the dynamic \ER graph lies in its relative analytic accessibility for also tackling more sophisticated topological issues.
Furthermore, in the same way that results proven for the static case have turned out to be of a `universal' nature regarding connectivity, in that they hold for far more complicated graphs and networks, we believe that the topological results of the paper have similar extensions.

\subsubsection{Clique complexes:}
\label{subsubsec:cc}
The study of the topology of \ER graphs typically revolves around the clique complexes that they generate, which we now define.

We first introduce the notion of an  {\it abstract simplicial complex} which is a purely combinatorial notion. A family $\cK$ of non-empty finite subsets of $V$ is an abstract simplicial complex if it is closed under the operation of taking non-empty subsets, i.e.,\  $\cY\subset \cX \in \cK\implies\cY\in \cK$. Elements of $\cK$ are called faces or simplices, and the dimension of a face $\cX$ is its cardinality $|\cX|$ minus 1.  Elements of dimension 0 are called vertices. The dimension of $\cK,$ denoted $\dim(\cK),$ is the supremum over dimensions of all its faces.

Abstract simplicial complexes also have concrete, geometric realisations in Euclidean space. In particular, if $\cK$ is finite, which is the only situation of interest to us, then this is simple. Firstly, embed the vertices of $\cK$ as an  affinely independent subset in $\real^N$, for sufficiently large $N$. For example, take $N$ to be the number of vertices, number the vertices $v_1,\dots,v_N$, write $e_j\in\real^N$ for the vector with a 1 in the $j$-th position and all other entries 0, and map $v_j\to e_j$.  Then any face $\cX\in \cK$ can be identified with the geometric simplex in $\real^N$ spanned by the corresponding embedded vertices. The geometric realisation is then the union of all such simplices.

Consider a (undirected) graph $G.$ Then a {\it clique} in $G$ is just a subset of vertices in $G$ such that each pair of vertices is joined by an edge. The {\it clique complex},  $\mathscr{X}(G)$, is the collection of all subsets of vertices that form a clique in $G$. Since a subset of a clique is itself a clique, $\mathscr{X}(G)$ is indeed an abstract simplicial complex. In the corresponding geometric realisation, each clique of $k$ vertices is represented by a simplex of dimension $k-1$. The 1-skeleton of  $\mathscr{X}(G)$  (which is the underlying graph of the complex) is a graph with a vertex for every 1-element set in $\mathscr{X}(G)$ and an edge for every 2-element set in $\mathscr{X}(G)$, and so  is isomorphic to $G$ itself.

Henceforth we will study the temporal evolution of the topology of the clique complexes generated from the dynamic \ER graph; viz.\ the sets
  \beq
  \mathscr{X}(n, p;t )\ := \ \mathscr{X}( G(n, p;t)).
  \label{equn:Xnpt}
  \eeq
In order to do this, we shall study the Betti numbers of these sets.

\subsubsection{Betti numbers:}
\label{subsubsec:Betti}
Throughout this paper we work with reduced Betti numbers and for notational convenience we shall drop the word reduced henceforth. There is really no good way to define Betti numbers in a few, self-contained, paragraphs. Formally, for an integer $k\geq 0$,  the $k$-th Betti number $\beta_k\equiv \beta_k(X)$ of a topological space $X$ is  the rank  of the abelian group $H_k(X, \mathbb{A})$, the reduced $k$-th homology group of $X$ with coefficients from the abelian group $\mathbb{A}.$ The reduced homology groups themselves  are the quotient groups $H_{k} = \ker \delta_{k} / \mathrm{Im} \; \delta_{k+1}$, where the  $\delta_{k}$'s  are the boundary maps for $X.$ In this paper, we assume $\mathbb{A} = \mathbb{Q},$ the field of rationals, consistent with  \cite{kahle2009topology,kahle2014sharp,kahle2013limit}.

The problem is that, as succinct as this description may be, it is of little help to a reader who has not already worked through one of the standard texts on Algebraic Geometry such as   \cite{hatcher606algebraic}, or perhaps the less standard \cite{edelsbrunner2010computational},  which is motivated by computational issues and  somewhat closer to the specific focus of the current paper.

Thus, we shall not attempt to define Betti numbers rigorously, but shall start with three examples and then allow some imprecision. For the following discussion, it is useful to assume that the topological space $X$ is a subset of some finite dimensional Euclidean space $\mathbb{R}^N.$ As for the examples,   $\beta_0(X)$ equals one less than the number of connected components in $X$.
$\beta_1(X)$ counts the number of 1-dimensional, or `topologically circular' holes - think of holes in a 2 or 3 dimensional object that you could poke a finger through. If $X$ is 3-dimensional, then $\beta_2(X)$ counts the number of `voids' within $X$ - think of the interior of a tennis ball, or of a  bagel that had an air pocket running around the entire ring. Higher order Betti numbers are rather harder to describe this way, since everyday language lacks the vocabulary needed to describe high dimensional objects. Roughly speaking, however, $\beta_k(X)$ counts the number of distinct regions in $X$ which  are `topologically equivalent to' the boundary of a solid, $k$-dimensional set, something which we refer to as a `$(k-1)$-cycle' below.  As such, increasing $k$ increases the qualitative level of topological complexity one is studying, while increasing $\beta_k$ for a fixed $k$  is an indication of quantitatively more complexity at the given level. This is true only  up to a point, since for all $k \geq N$, $\beta_k(X)\equiv 0$. Fortunately, at least in order to understand the thrust of the main results of this paper, these necessarily imprecise descriptions of Betti numbers should suffice.

The results of this paper concentrate on the $n\to\infty$ asymptotic behaviour of  stochastic processes describing the normalised Betti numbers of the clique complexes associated with the dynamic \ER graphs; viz.
\beq
\label{equn:barbeta}
\bar \beta_{n,k}(t) \ := \ \dfrac{\beta_{n, k}(t) - \mathbb{E}[\beta_{n, k}(t)]}{\sqrt{{{\rm{Var}}}[\beta_{n, k}(t)]}},
\eeq
where
\beq
\label{equn:betank}
\beta_{n,k}(t) \ := \ \beta_k \left( \mathscr{X}(n, p;t)\right)\ = \ \beta_k\left(\mathscr{X}\left(G(n,p;t)\right)       \right).
\eeq

\subsection{Results}

\subsubsection{\ER graphs and associated topology:}
\label{subsubsection:ERT}

The topological study of static random graphs and their associated simplicial complexes, beyond  classical issues of connectivity and degree, has seen considerable recent activity, including \cite{babson2011fundamental, kahle2007neighborhood, kahle2011random, kahle2013limit, kahle2014sharp, linial2006homological, meshulam2009homological}. A recent, well motivated review is  \cite{kahle2014topology}. Most of this literature follows the theme that  Betti numbers of increasing index are good quantifiers of topological complexity, and so are the appropriate measure to study.

In terms of the  (static) \ER graph, heuristics imply that for small $p$ the associated clique complex will, with high probability, be topologically simple, but that the complexity will grow with increasing $p$. Thinking a little more deeply,  as $p$ grows the clique complex changes from a collection of disconnected vertices (so that $\beta_0$ is large) to a highly connected object (so that, at full connectivity, $\beta_0$ drops to its minimum value of 0). At about the same stage,   simple, 1-dimensional cycles start forming (so that $\beta_1$ grows) until these cycles fill in and then
to produce empty tetrahedra-type objects (so that $\beta_1$ drops while $\beta_2$ grows).   The following result, which combines the result from \cite[Theorem 1.1]{kahle2014sharp} and the discussion below (1) in  \cite{kahle2014erratum}, confirms  this description.

\begin{theorem}[\cite{kahle2014sharp,kahle2014erratum}]
\label{thm:VanishingOfAllButOneHomology}
Fix $k \geq 1, \ M>0$ and $t \geq 0.$ Let $p = n^\alpha$, $\alpha \in \left(-\tfrac{1}{k}, -\tfrac{1}{k + 1}\right)$. Then,
\[
\underset{n \rightarrow  \infty}{\lim}\Pr\{\beta_{n, k}(t) \neq 0,\, \beta_{n,j}(t) \ = \ 0, \, \forall j \neq k \} = 1 - o(n^{-M}).
\]
\end{theorem}

Since $G(n, p, t)$ is distributed as a \ER graph, the above result is a simple rephrasing of the original result given in \cite{kahle2014sharp,kahle2014erratum}. This result shows that there is a sequence of clearly marked phase transitions, and between each of these there is a dominant Betti number, and so a dominant type of homology in the clique complex.
Of more interest to us, however, is the following central limit theorem that is a consequence of  \cite[Theorem 2.4]{kahle2013limit}  and \cite[Theorem 1.1]{kahle2014erratum}.

\begin{theorem}[\cite{kahle2013limit,kahle2014erratum}]
\label{thm:CentralLimitTheoremBettiNumbers}
Fix $k \geq 1$, $t \geq 0,$ and let $p$ be as in Theorem~\ref{thm:VanishingOfAllButOneHomology}. Then, as $n \rightarrow \infty$,\[
\dfrac{\beta_{n, k}(t) - \mathbb{E}[\beta_{n, k}(t)]}{\sqrt{{{\rm{Var}}}[\beta_{n, k}(t)]}} \ \Rightarrow\  \mathcal{N}(0,1),
\]
 where $\mathcal{N}(0, 1)$ denotes a standard Gaussian  and $\Rightarrow$ denotes convergence in distribution.
\end{theorem}

\subsubsection{Dynamic \ER graphs and associated topology:}
\label{subsubsection:DERT}
The main result of this paper is the following extension of Theorem~\ref{thm:CentralLimitTheoremBettiNumbers}.

\begin{theorem}
\label{thm:ConvergenceBettiNumberOrnsteinUhlenbeckProcess}
Fix $k \geq 1$,  $\lambda>0$. Let $p = n^\alpha$, $\alpha \in \left(-\tfrac{1}{k}, -\tfrac{1}{k + 1}\right)$. Then, as $n \rightarrow \infty$,
\[
\{\bar{\beta}_{n, k}(t) : t\geq 0\}
\ \Rightarrow \ \{\mathcal{U}_{\lambda} : t \geq 0\}
\]
\end{theorem}
where $\{\mathcal{U}_{\lambda}(t): t \geq 0\}$ is the stationary, zero mean,  Ornstein-Uhlenbeck process  with covariance ${\rm{Cov}}[\mathcal{U}_{\lambda}(t_1), \mathcal{U}_{\lambda}(t_2)] = e^{-\lambda |t_1 - t_2|}$,  and here $\Rightarrow$  denotes convergence in distribution on the Skorokhod space of functions on $[0,\infty)$.

Although, in view of Theorem \ref{thm:CentralLimitTheoremBettiNumbers}, it is not surprising that the limits of the random processes $\bar{\beta}_{n,k}$ are Gaussian, it is somewhat surprising that, as Ornstein-Uhlenbeck processes,  they are Markovian.  While the underlying dynamic \ER process is Markovian, this is not the case for the processes $\bar{\beta}_{n,k}$, as shown in Appendix
\ref{app:markov}.
%This discovery  alone makes the effort of proving Theorem \ref{thm:ConvergenceBettiNumberOrnsteinUhlenbeckProcess} worthwhile.

\subsubsection{On proving Theorem \ref{thm:ConvergenceBettiNumberOrnsteinUhlenbeckProcess}:}
\label{subsec:IdeaOfProof}

Since working directly with Betti numbers is difficult, we adopt the approach of \cite{kahle2009topology,kahle2013limit}. Let  $f_{n, k}(t)$ denote the number of $(k + 1)$-cliques
in $G(n, p, t)$ and let
\begin{equation}
\label{eqn:defnEulerCharacteristicCliqueCount}
\chi_{n}(t) \ := \ \sum_{j = 0}^{n - 1} (-1)^{j} f_{n, j} (t) \ =\  \sum_{j = 0}^{n - 1} (-1)^{j} \beta_{n, j} (t)
\end{equation}
be the Euler-Poincar{\' e} characteristic of $\mathscr{X}(n, p, t);$ see \cite[p101]{edelsbrunner2010computational} for details.  Define
\begin{equation}
\label{eqn:ScaledShiftedCliqueCountProcess}
\bar{f}_{n, k}(t) \ := \  \frac{f_{n, k}(t) - \mathbb{E}[f_{n, k}(t)]}{\sqrt{{\rm {Var}}[f_{n, k}(t)]}},
\quad\text{and}\quad
%\label{eqn:ScaledShiftedEulerCharacteristicProcess}
\bar{\chi}_{n}(t) \ := \ \frac{\chi_{n}(t) - \mathbb{E}[\chi_{n}(t)]}{\sqrt{{{\rm{Var}}}[\chi_{n}(t)]}}.
\end{equation}

We first establish weak convergence for $\{\bar{f}_{n, k}(t) : t\geq 0\}$. Using the first equality in \eqref{eqn:defnEulerCharacteristicCliqueCount}, we then establish weak convergence for $\{\bar{\chi}_{n}(t) : t\geq 0\}$. Finally, Theorem~\ref{thm:ConvergenceBettiNumberOrnsteinUhlenbeckProcess} is proven using the second equality in \eqref{eqn:defnEulerCharacteristicCliqueCount} and Theorem~\ref{thm:VanishingOfAllButOneHomology}.

To carry this out, in Section~\ref{sec:MathematicalBackground} we  quote  some  results on the convergence  of random variables and processes. In Section~\ref{sec:PreliminaryResults}, we discuss some preliminary results concerning the mean and variance of $f_{n, k}(t)$, $\chi_{n}(t)$, and $\beta_{n, k}(t)$. The covariance functions of the processes  $\bar{f}_{n, k}$, $\bar{\chi}_{n}$, and
 $\bar\beta_{n, k}$ are derived in Section~\ref{sec:Covariance} and exploited in Section
  \ref{sec:FiniteDimensionalDistribution} to establish    convergence of the finite dimensional distributions of the  $\bar{\beta}_{n, k}$. In Section~\ref{sec:Tightness}, we establish tightness for the processes $\bar{\beta}_{n, k}$, and complete the proof of Theorem
\ref{thm:ConvergenceBettiNumberOrnsteinUhlenbeckProcess}.

\section{On convergence in distribution}
\label{sec:MathematicalBackground}

To help the reader and make this paper a little more self-contained, we now quote two theorems about weak convergence. The first, from \cite{barbour1989central},
is a central limit theorem for dissociated random variables (defined formally in the statement of Theorem \ref{thm:WassersteinMetricUB}).  The second,
 which comes from  combining  Theorems 7.8, 8.6, and 8.8 of \cite{ethier2009markov}, is about convergence, in the Skorokhod space,  to the stationary Ornstein-Uhlenbeck process.

%\subsection{Normal approximation theorem}
%\label{subsubsec:NormalApproximationTheorem}

Before stating the theorems, we remind the reader of the definition of the  $L_1$-Wasserstein metric for real valued random variables.
%\begin{definition}
%\label{defn:L1Wasserstein}
For two real valued random variables  $Y_1$ and $Y_2$, their  $L_1$-Wasserstein distance is
\[
d_1(Y_1, Y_2) = \sup_{\psi} |\mathbb{E}[\psi(Y_1)] - \mathbb{E}[\psi(Y_2)]|,
\]
where the $\sup$ is over all functions $\psi: \mathbb{R} \rightarrow \mathbb{R}$ with $\sup_{y_1 \neq y_2}{|\psi(y_1) - \psi(y_2)|)/|y_1 - y_2|} \leq 1$.
Recall also that convergence in this metric implies convergence in distribution.

\begin{theorem}[\cite{barbour1989central}]
\label{thm:WassersteinMetricUB}
Let $\{Y_i : i \equiv (i_1, \ldots, i_r) \in I\},$ for some index set $I$ of $r-$tuples, be a sequence of dissociated random variables. That is, for any $J, L \subseteq I,$ $\{Y_i: i \in J\}$ and $\{Y_i: i \in L\}$ are independent whenever $(\bigcup_{i \in J} \{i_1, \ldots, i_r\}) \cap  (\bigcup_{i \in L} \{i_1, \ldots, i_r\}) = \emptyset.$ Let $\mathscr{W} = \sum_{i \in I} Y_i$ and, for each $i \in I,$ let $\cY(i) := \{k \in I: \{k_1, \ldots, k_r\} \cap \{i_1, \ldots, i_r\} \neq \emptyset \}$ be the dependency neighbourhood of $i.$ If $\mathbb{E}[Y_i] \equiv 0$ and ${{\rm{Var}}}[\mathscr{W}] = 1,$ then there exists a universal constant $\rho > 0$ such that
\begin{equation}
\label{eqn:WassersteinMetricUB}
d_1(\mathscr{W}, \mathcal{N}(0,1)) \leq \rho \sum_{i \in I} \sum_{j, \ell \in \mathcal{Y}(i)}\mathbb{E}\big[|Y_iY_jY_\ell|\big] + \mathbb{E} \big[|Y_i Y_j|\big]\; \mathbb{E}\big[|Y_{\ell}|\big].
\end{equation}
\end{theorem}

\eqref{eqn:WassersteinMetricUB} is obtained by combining Theorem 1 and (2.7) in \cite{barbour1989central} (see also discussion above (2.7) in \cite{barbour1989central}). Let $D_{\mathbb{R}}[0, \infty)$ denote the (Skorokhod) space of  right continuous functions on $[0,\infty$) with left limits, and   write $\hat{d}$ for the usual  (Skorokhod) metric on this space.

\begin{theorem}[\cite{ethier2009markov}]
\label{thm:SufficientConditionsConvergenceOrnsteinUhlenbeckProcess}
Let $\{X_n(t): t \geq 0\} $, $n\geq 1$,  be  a sequence of $(D_{\mathbb{R}}[0, \infty), \hat{d})$ valued stochastic processes  satisfying the following conditions:

{\renewcommand*\theenumi{$\pmb{\mathfrak{C}}_\arabic{enumi}$}
\begin{itemize}
\item \label{cond: Gaussian} Convergence of finite dimensional distributions: For any  $t_1, \ldots, t_m \geq 0$,
\[
(X_n(t_1), \ldots, X_n(t_m)) \Rightarrow (\mathcal{U}_{\lambda}(t_1), \ldots, \mathcal{U}_{\lambda}(t_m)) \text{ as $n \to \infty.$}
\]
\item Tightness: The sequence $\{\{X_{n}(t) : t\geq 0\} : n \geq 1\}$ is tight, for which it is sufficient that the following two conditions hold.
\begin{enumerate}
\item \label{cond: InitialDrift} There exists $\Upsilon > 0$ such that
\[
\lim_{\delta \rightarrow 0} \limsup_{n \rightarrow \infty} \mathbb{E}|X_n(\delta) - X_n(0)|^\Upsilon = 0.
\]

\item \label{cond: Tightness} For each $T > 0$,  there exist constants $\Upsilon_1 > 0$, $\Upsilon_2 > 1$, and $K > 0$ such that, for all $n$, $0 \leq t \leq T + 1$,
 and $0 \leq h \leq t$,
\[
\mathbb{E}\left[|X_n(t +h) - X_n(t)|^{\Upsilon_1} |X_n(t) - X_n(t - h)|^{\Upsilon_1} \right] \leq Kh^{\Upsilon_2}.
\]
\end{enumerate}
\end{itemize}}
Then $\{X_n(t) : t\geq 0\} \Rightarrow \{\mathcal{U}_{\lambda(t)} : t\geq 0\}$ as $n \rightarrow \infty$, where $\Rightarrow$  denotes convergence on the Skorokhod space.
\end{theorem}

\section{Preliminary Results}
\label{sec:PreliminaryResults}

We study here the asymptotic variances of $f_{n, k}(t), \chi_{n}(t)$, and $\beta_{n, k}(t).$ Due to stationarity, these variances are independent of $t$. We  start with some notation.

We   write $[n]:= \{1, \ldots, n\}$ for the vertex set of the dynamic \ER graph. This is not dependent on $t$.
We  write $\binom{[n]}{j + 1}$ to denote the collection of all subsets of $[n]$ of size $j + 1$, while $\binom{n}{j + 1}$ is the usual binomial coefficient. For $A \in \binom{[n]}{j + 1}$, let $1_A(t)$ be the indicator function for $A$ being  a $(j + 1)$-clique in $G(n, p, t)$. We can now write
\begin{equation}
\label{eqn:CliqueCount}
f_{n, j}(t) = \sum_{A \in \binom{[n]}{j + 1}} 1_A(t),
\end{equation}
from which it immediately follows that
\begin{equation}
\label{eqn:MeanCliqueCount}
\mathbb{E}[f_{n, j}(t)] = \tbinom{n}{j + 1} p^{\tbinom{j + 1}{2}}
\end{equation}
and
\begin{eqnarray*}
\mathbb{E}[f^2_{n, j}(t)] & = & \sum_{A_1 \in \tbinom{[n]}{j + 1}} \sum_{A_2 \in \tbinom{[n]}{j + 1}} \mathbb{E}[1_{A_1}(t)1_{A_2}(t)] \\
& = & \tbinom{n}{j + 1} \sum_{A_2 \in \tbinom{[n]}{j + 1}} \mathbb{E}[1_{A_1}(t)1_{A_2}(t)]\\
& = & \tbinom{n}{j + 1}\sum_{i = 0}^{j + 1}\tbinom{j + 1}{i} \tbinom{n - j - 1}{j + 1 - i} \tfrac{p^{2 \binom{j + 1}{2}}}{p^{\binom{i}{2}}},
\end{eqnarray*}
where in the second equality $A_1$ is an arbitrary but fixed $k$-face. The second equality follows because the inner sum on the right hand side  is the same  for each $A_1$, and the third equality follows by combining all faces $A_2$ that share $i$ vertices with $A_1$. Hence,
\begin{equation}
\label{eqn:VarianceCliqueCount}
{{\rm{Var}}}[f_{n, j }(t)] = \tbinom{n}{j + 1}\sum_{i = 0}^{j + 1}\tbinom{j + 1}{i} \tbinom{n - j - 1}{j + 1 - i} \tfrac{p^{2 \binom{j + 1}{2}}}{p^{\binom{i}{2}}} - \tbinom{n}{j + 1}^2 p^{2 \binom{j + 1}{2}}.
\end{equation}

The below result gives the behaviour of ${{\rm{Var}}}[f_{n, j}(t)]$ as $n \rightarrow \infty$ for different $j$.

\begin{lemma}
\label{lem:VarianceOfAllCliqueCounts}
Fix $k \geq 1$, $j \geq 0$, and $t \geq 0$. Let $p = n^\alpha$, $\alpha \in \left(-\tfrac{1}{k}, -\tfrac{1}{k + 1}\right)$.
\begin{enumerate}
\item[(i)] ${{\rm{Var}}}[f_{n, 0}(t)] \equiv 0$.
\item[(ii)]  If $j = 2k - 1$ and $\alpha \in \big[-\tfrac{1}{k + 0.5}, -\tfrac{1}{k + 1} \big)$, or if $1\leq j \leq 2k - 2$,  then
\[
{\rm{Var}}[f_{n, j}(t)] \leq  2^{j + 1} n^{2j} p^{2\binom{j + 1}{2} - 1}.
\]
\item[(iii)] If $j = 2k - 1$ and $\alpha \in \left(-\tfrac{1}{k}, -\tfrac{1}{k + 0.5}\right]$, or if $j \geq 2k$, then
\[
{\rm{Var}}[f_{n, j}(t)] \leq 2^{j + 1} n^{j + 1}p^{\binom{j + 1}{2}}.
\]
\end{enumerate}
\end{lemma}
\begin{proof}
The first claim is trivial since $f_{n, 0}(t) \equiv n.$ So we prove only the other two.
Since $\binom{n}{j + 1} = \sum_{i = 0}^{j + 1}\binom{j + 1}{i}\binom{n - j - 1}{j + 1 - i}$, it follows from \eqref{eqn:VarianceCliqueCount} that
\begin{equation}
\label{eqn:VarianceCliqueCountSimplifiedExp}
{\rm{Var}}[f_{n, j}(t)]  =   \sum_{i = 2}^{j + 1}  \tbinom{j + 1}{i} \tbinom{n}{j + 1} \tbinom{n - j - 1}{j + 1 - i} \left[p^{2\binom{j + 1}{2} - \binom{i}{2}} - p^{2\binom{j + 1}{2}}\right].
\end{equation}
The summation starts from $2$ because the term in the square brackets above is zero for $i = 0,1$. Note that $\tbinom{n}{j + 1} \tbinom{n - j - 1}{j + 1 - i} \leq n^{2 j + 2 - i}$. Further, $p = n^\alpha$ with $\alpha < 0$. Hence the term inside the square bracket is positive for each $i$, and bounded from above by $p^{2\tbinom{j + 1}{2} - \tbinom{i}{2}}$. Hence,  to prove the desired result, it suffices to obtain bounds for $\sum_{i = 2}^{j + 1} \tbinom{j + 1}{i} n^{\zeta_j(i)}$, where $\zeta_j(i) = 2j + 2 - i + \alpha [2\binom{j + 1}{2} - \binom{i}{2}]$.

As $\alpha < 0$, $\zeta_j$ is a convex function. Hence, one of $\zeta_j(2)$ or $\zeta_j(j + 1)$ maximizes $\zeta_j(i)$ for $i \in \{2, \ldots, j + 1\}$. When the conditions of (ii) hold, $\zeta_j(2) \geq \zeta_j(j + 1)$. Similarly, when the conditions of (iii) hold, $\zeta_j(j + 1) \geq \zeta_j(2)$. At $\alpha =  \tfrac{1}{k + 0.5}$, $\zeta_j(2) = \zeta_j(j+1)$. Since $\sum_{i = 2}^{j + 1} \tbinom{j + 1}{i} = 2^{j + 1}$, the desired result is now easy to see.
\end{proof}

Let $p$ be as in Lemma~\ref{lem:VarianceOfAllCliqueCounts}. The next result computes the exact order of ${\rm{Var}}[f_{n, k}(t)]$.

\begin{lemma}
\label{lem:VarianceCliqueCount}
Fix $k \geq 1$.  Let $p = n^\alpha$, $\alpha \in \left(-\tfrac{1}{k}, -\tfrac{1}{k + 1}\right)$. Then, for each $t \geq 0$,
\[
{\rm{Var}}[f_{n, k}(t)] = \Theta (n^{2k}p^{2\binom{k + 1}{2} - 1}).
\]
\end{lemma}
\begin{proof}
Recall from  \eqref{eqn:VarianceCliqueCountSimplifiedExp} that
\[
{\rm{Var}}[f_{n, k}(t)]  =   \sum_{i = 2}^{k + 1}  \tbinom{k + 1}{i} \tbinom{n}{k + 1} \tbinom{n - k - 1}{k + 1 - i} \left[p^{2\binom{k + 1}{2} - \binom{i}{2}} - p^{2\binom{k + 1}{2}}\right].
\]
Observe that $2\binom{k + 1}{2} - \binom{i}{2} < 2\binom{k + 1}{2}$ for each $i \in \{2, \ldots, k + 1\}$. Hence to prove the desired result, it suffices to show that
\[
\sum_{i = 2}^{k + 1}  \tbinom{k + 1}{i} \tbinom{n}{k + 1} \tbinom{n - k - 1}{k + 1 - i} \left[p^{2\binom{k + 1}{2} - \binom{i}{2}}\right] = \Theta (n^{2k}p^{2\binom{k + 1}{2} - 1}).
\]
Since $\tbinom{n}{k + 1} \tbinom{n - k - 1}{k + 1 - i} = \Theta(n^{2k + 2 - i})$, arguing as in the proof of Lemma~\ref{lem:VarianceOfAllCliqueCounts}, the above claim is easy to see, and the result follows.
\end{proof}

The following result is now immediate from Lemmas~\ref{lem:VarianceOfAllCliqueCounts} and \ref{lem:VarianceCliqueCount}.
\begin{corollary}
\label{cor:VarjVarkRatio}
Fix $k \geq 1$, $j \geq 0$, and $t \geq 0$. Let $p = n^\alpha$, $\alpha \in \left(-\tfrac{1}{k}, -\tfrac{1}{k + 1}\right).$ Then
\[
\lim_{n \rightarrow \infty} \frac{{\rm{Var}}[f_{n, j}(t)]}{{\rm{Var}}[f_{n, k}(t)]} = 0 \text{ whenever $j \neq k.$}
\]
\end{corollary}

We next compare ${\rm{Var}}[(-1)^k\chi_{n}(t) - f_{n, k}(t)]$, ${\rm{Var}}[\chi_{n}(t)]$ with ${\rm{Var}}[f_{n, k}(t)]$ as $n \rightarrow \infty$.
\begin{lemma}
\label{lem:AsymptoticDifferenceCliqueCountEulerCharacterisitc}
Fix $k \geq 1$. Let $p = n^\alpha$, $\alpha \in \left(-\tfrac{1}{k}, -\tfrac{1}{k + 1}\right)$. Then, for each $t \geq 0$,
\[
\lim_{n \rightarrow \infty}\frac{{\rm{Var}}[(-1)^k\chi_{n}(t) - f_{n, k}(t)]}{{\rm{Var}}[f_{n, k}(t)]} = 0.
\]
\end{lemma}
\begin{proof}
From \eqref{eqn:defnEulerCharacteristicCliqueCount}, we have
\begin{eqnarray}
& & {\rm{Var}}[(-1)^k\chi_{n}(t) - f_{n, k}(t)] \nonumber \\
& \leq & \sum_{0 \leq j \leq n - 1, \; j \neq k} {\rm{Var}}[f_{n, j}(t)]  + 2 \sum_{0 \leq i < j \leq n - 1; \; i,j \neq k} |{\rm{Cov}}[f_{n, i}(t), f_{n, j}(t)]| \nonumber \\
& = & 2 \sum_{0 \leq i \leq j \leq (n - 1); \; i, j \neq k}\sqrt{{\rm{Var}}[f_{n, i}(t)] {\rm{Var}}[ f_{n, j}(t)]}, \nonumber\\
& \leq & 2 \sum_{0 \leq i \leq j \leq 4k + 4; \; i, j \neq k} \sqrt{{\rm{Var}}[f_{n, i}(t)] {\rm{Var}}[ f_{n, j}(t)]} \nonumber \\
& & + \; 2 \sum_{0 \leq i \leq (n - 1), \; i \neq k} \; \sum_{4k + 5 \leq j \leq (n - 1)} \sqrt{{\rm{Var}}[f_{n, i}(t)] {\rm{Var}}[ f_{n, j}(t)]}. \label{eqn:BoundEulerCliqueCountDiff}
\end{eqnarray}
Let $n$ be sufficiently large. From Lemma~\ref{lem:VarianceOfAllCliqueCounts}, note that, for $j \geq 2k$,
\[
{\rm{Var}}[f_{n, j}(t)] \leq 2^{j + 1} n^{j + 1} p^{\tbinom{j + 1}{2}} \leq 2n^{2j + 1 + \alpha \tbinom{j + 1}{2}}.
\]
But for all $j \geq 2k + 1$, $2j + 1 + \alpha \tbinom{j + 1}{2}$ monotonically decreases with $j.$ Hence
\[
{\rm{Var}}[f_{n, j}(t)] \leq 2 n^{2(4k + 5) + 1  + \alpha \tbinom{(4k + 5) + 1}{2}}
\]
for all $j \geq 4k + 5.$ This implies that
\begin{eqnarray}
& & \sum_{0 \leq i \leq (n - 1), i \neq k} \;  \sum_{4k + 5 \leq j \leq (n - 1)} \sqrt{{\rm{Var}}[f_{n, i}(t)]}\sqrt{{\rm{Var}}[f_{n, j}(t)]} \nonumber \\
& \leq & n \sqrt{2 n^{2(4k + 5) + 1  + \alpha \tbinom{(4k + 5) + 1}{2}}} \sum_{0 \leq i \leq (n - 1), i \neq k} \sqrt{{\rm{Var}}[f_{n, i}(t)]} \nonumber \\
& \leq & n \sqrt{2 n^{2(4k + 5) + 1  + \alpha \tbinom{(4k + 5) + 1}{2}}} \sum_{0 \leq i \leq 4k + 4, i \neq k} \sqrt{{\rm{Var}}[f_{n, i}(t)]} \nonumber \\
& & + \; n \sqrt{2 n^{2(4k + 5) + 1  + \alpha \tbinom{(4k + 5) + 1}{2}}} \sum_{4k + 5 \leq i \leq (n - 1)} \sqrt{{\rm{Var}}[f_{n, i}(t)]}\nonumber \\
& \leq & n \sqrt{2 n^{2(4k + 5) + 1  + \alpha \tbinom{(4k + 5) + 1}{2}}} \sum_{0 \leq i \leq 4k + 4, i \neq k} \sqrt{{\rm{Var}}[f_{n, i}(t)]} \nonumber \\
& & + \; n^2 \left[2 n^{2(4k + 5) + 1  + \alpha \tbinom{(4k + 5) + 1}{2}} \right]. \label{eqn:nSqTermsBound}
\end{eqnarray}
Observe that
\[
\lim_{n \to \infty} n^2 \left[\frac{n^{2(4k + 5) + 1} p^{\tbinom{(4k + 5) + 1}{2}}}{n^{2k} p^{2\tbinom{k + 1}{2} - 1}}\right] = 0.
\]
Combining this, \eqref{eqn:nSqTermsBound}, Lemma~\ref{lem:VarianceCliqueCount}, and Corollary~\ref{cor:VarjVarkRatio}, it follows that
\[
\lim_{n \to \infty} \frac{\sum_{0 \leq i \leq (n - 1), i \neq k} \;  \sum_{4k + 5 \leq j \leq (n - 1)} \sqrt{{\rm{Var}}[f_{n, i}(t)]}\sqrt{{\rm{Var}}[f_{n, j}(t)]}}{{\rm{Var}}[f_{n, k}(t)]} = 0.
\]
Similarly, from Corollary~\ref{cor:VarjVarkRatio}, we have
\[
\lim_{n \to \infty} \frac{\sum_{0 \leq i \leq j \leq 4k + 4; \; i, j \neq k} \sqrt{{\rm{Var}}[f_{n, i}(t)]}\sqrt{{\rm{Var}}[f_{n, j}(t)]}}{{\rm{Var}}[f_{n, k}(t)]} = 0.
\]
Combining the above two relations with \eqref{eqn:BoundEulerCliqueCountDiff}, the desired result is easy to see.
\end{proof}

\begin{lemma}
\label{lem:AsymptoticVarianceEulerCharacteristicCliqueCountSame}
Fix $k \geq 1$. Let $p = n^\alpha$, $\alpha \in \left(-\tfrac{1}{k}, -\tfrac{1}{k + 1}\right)$. Then, for each $t \geq 0$,
\[
\lim_{n \rightarrow \infty}\frac{{\rm{Var}}[\chi_{n}(t)]}{{\rm{Var}}[f_{n, k}(t)]} = 1.
\]
\end{lemma}
\begin{proof}
By adding and subtracting $f_{n, k}(t)$, we have
\begin{eqnarray*}
{\rm{Var}}[\chi_{n}(t)] & = & {\rm{Var}}[f_{n, k}(t)] + {\rm{Var}}[(-1)^k\chi_{n}(t) - f_{n, k}(t)] \\
& & \qquad\qquad +\;  2 {\rm{Cov}}[(-1)^k\chi_{n}(t) - f_{n, k}(t), f_{n, k}(t)].
\end{eqnarray*}
Hence it follows that
\[
\left|\frac{{\rm{Var}}[\chi_{n}(t)]}{{\rm{Var}}[f_{n, k}(t)]} - 1 \right| \leq \frac{{\rm{Var}}[(-1)^k\chi_{n}(t) - f_{n, k}(t)]}{{\rm{Var}}[f_{n, k}(t)]} + 2\sqrt{\frac{{\rm{Var}}[(-1)^k\chi_{n}(t) - f_{n, k}(t)]}{{\rm{Var}}[f_{n, k}(t)]}}.
\]
The desired result now follows from Lemma~\ref{lem:AsymptoticDifferenceCliqueCountEulerCharacterisitc}.
\end{proof}

In a similar spirit to the above two results, Lemmas~\ref{lem:AsymptoticDifferenceBettiNumberEulerCharacterisitc} and \ref{lem:AsymptoticVarianceBettiNumberEulerCharacteristicSame} given below compare the limiting behaviour of ${\rm{Var}}[\chi_{n}(t)]$ with ${\rm{Var}}[\beta_{n, k}(t)]$. These results  are be due to Kahle and Meckes in \cite{kahle2014erratum}.  (The results there were established for \ER graphs and hence are applicable in our setup to $G(n, p, t)$ for any fixed $t$. Their notations $\beta_k$ and $\tilde{\beta}_{k}$ correspond to  $\beta_{n, k}(t)$ and $(-1)^k \chi_{n}(t)$ in our context.)

\begin{lemma}[\cite{kahle2014erratum}]
\label{lem:AsymptoticDifferenceBettiNumberEulerCharacterisitc}
Fix $k \geq 1$. Let $p = n^\alpha$, $\alpha \in \left(-\tfrac{1}{k}, -\tfrac{1}{k + 1}\right)$. Then, for each $t \geq 0$,
\[
\lim_{n \rightarrow \infty}\frac{{\rm{Var}}[\beta_{n, k}(t) - (-1)^k\chi_{n}(t)]}{{\rm{Var}}[\chi_{n}(t)]} = 0.
\]
\end{lemma}

\begin{lemma}[\cite{kahle2014erratum}]
\label{lem:AsymptoticVarianceBettiNumberEulerCharacteristicSame}
Fix $k \geq 1$. Let $p = n^\alpha$, $\alpha \in \left(-\tfrac{1}{k}, -\tfrac{1}{k + 1}\right)$. Then, for each $t \geq 0$,
\[
\lim_{n \rightarrow \infty}\frac{{\rm{Var}}[\beta_{n, k}(t)]}{{\rm{Var}}[\chi_{n}(t)]} = 1.
\]
\end{lemma}

\section{Covariance}
\label{sec:Covariance}
In this section we investigate the
 covariance functions of the processes \CliqueBar, \! \EulerBar, and \BetaBar \! as $n \to \infty.$ We shall need these in Section~\ref{sec:FiniteDimensionalDistribution} to show that finite dimensional distributions of \BetaBar  \! converge to those of the stationary Ornstein-Uhlenbeck process.

\begin{lemma}
\label{lem:AsymptoticCovarianceCliqueCount}
Fix $k \geq 1$. Let $p = n^\alpha$, $\alpha \in \left(-\tfrac{1}{k}, -\tfrac{1}{k + 1}\right)$. Then, for any $t_1, t_2 \geq 0$,
\[
\underset{n \rightarrow \infty}{\lim}{\rm{Cov}}[\bar{f}_{n, k}(t_1), \bar{f}_{n, k}(t_2)] = e^{-\lambda |t_1 - t_2|}.
\]
\end{lemma}
\begin{proof}
Fix arbitrary $t_1, t_2 \geq 0$, and define $L = e^{-\lambda |t_1 - t_2|}$. Using \eqref{eqn:CliqueCount}, note that
\begin{eqnarray*}
\mathbb{E} [f_{n, k}(t_1)  f_{n, k}(t_2)] & = & \sum_{A_1 \in \tbinom{[n]}{k + 1}} \sum_{A_2 \in \tbinom{[n]}{k + 1}} \mathbb{E}[1_{A_1}(t_1) 1_{A_2}(t_2)]\\
& = & \tbinom{n}{k + 1} \sum_{A_2 \in \tbinom{[n]}{k + 1}} \mathbb{E}[1_{A_1}(t_1) 1_{A_2}(t_2)],
\end{eqnarray*}
where in the second equality $A_1$ is an arbitrary, but  fixed, $k$-face. Rewriting the above  in terms of the number of  vertices common to $A_1$ and $ A_2,$ applying\eqref{eqn:EdgeTransition11}, \eqref{eqn:EdgeProbability} gives
\[
\mathbb{E} [f_{n, k}(t_1)  f_{n, k}(t_2)] =   \tbinom{n}{k + 1} \sum_{i = 0}^{k + 1} \tbinom{k + 1}{i} \tbinom{n - k - 1}{k + 1 - i} \tfrac{p^{2 \binom{k + 1}{2}}}{p^{\binom{i}{2}}} [p + (1 - p) L]^{\binom{i}{2}}.
\]
Combining this with \eqref{eqn:MeanCliqueCount} and \eqref{eqn:VarianceCliqueCount}, it is easy to see that
\[
{\rm{Cov}}[\bar{f}_{n, k}(t_1), \bar{f}_{n, k}(t_2)] = \frac{\sum\limits_{i = 0}^{k + 1} \binom{k + 1}{i} \binom{n - k - 1}{k + 1 - i} [1 + \frac{1 - p}{p}L]^{\binom{i}{2}} - \binom{n}{k + 1}}{ \sum\limits_{i = 0}^{k + 1}\binom{k + 1}{i} \binom{n - k - 1}{k + 1 - i} p^{- \binom{i}{2}} - \binom{n}{k + 1}}.
\]
Now using the fact that $\binom{n}{k + 1} = \sum_{i = 0}^{k + 1}\binom{k + 1}{i}\binom{n - k - 1}{k + 1 - i}$, we have
\[
{\rm{Cov}}[\bar{f}_{n, k}(t_1), \bar{f}_{n, k}(t_2)] = \frac{\sum\limits_{i = 2}^{k + 1} \binom{k + 1}{i} \binom{n - k - 1}{k + 1 - i}[(1 + \frac{1 - p}{p}L)^{\binom{i}{2}} - 1]}{\sum\limits_{i = 2}^{k + 1} \binom{k + 1}{i} \binom{n - k - 1}{k + 1 - i}[\frac{1 - p^{\binom{i}{2}}}{p^{\binom{i}{2}}}]}.
\]
By expanding terms inside the square brackets and cancelling out $\frac{1 - p}{p}$, we have that
\[
{\rm{Cov}}[\bar{f}_{n, k}(t_1), \bar{f}_{n, k}(t_2)] = L \frac{\sum\limits_{i = 2}^{k + 1}\binom{k + 1}{i} \binom{n - k - 1}{k + 1 - i}\left[\sum\limits_{j = 1}^{\binom{i}{2}}c_{ij} \left(\frac{1 - p}{p} L \right)^{j - 1}\right]}{\sum\limits_{i = 2}^{k + 1}\binom{k + 1}{i} \binom{n - k - 1}{k + 1 - i}\left[\sum\limits_{j = 1}^{\binom{i}{2}} \left(\frac{1}{p}\right)^{j - 1}\right]},
\]
where $c_{ij} = \binom{\binom{i}{2}}{j}$. Now observe that the term corresponding to $i = 2$ inside the summation in both the numerator as well as denominator is the same. Hence,
\[
{\rm{Cov}}[\bar{f}_{n, k}(t_1), \bar{f}_{n, k}(t_2)] = L + \frac{L\sum\limits_{i = 3}^{k + 1}\binom{k + 1}{i} \binom{n - k - 1}{k + 1 - i}\sum\limits_{j = 1}^{\binom{i}{2}}\left[\frac{c_{ij} \left((1 - p) L \right)^{j - 1} - 1}{p^{j - 1}}\right]}{\sum\limits_{i = 2}^{k + 1}\binom{k + 1}{i} \binom{n - k - 1}{k + 1 - i}\left[\sum\limits_{j = 1}^{\binom{i}{2}} \left(\frac{1}{p}\right)^{j - 1}\right]} := L (1 + \mathscr{Z}_{n, k}).
\]

To prove the desired result, it suffices to show that $\mathscr{Z}_{n, k} \rightarrow 0$ as $n \rightarrow \infty$. If $k = 1, $ then $\mathscr{Z}_{n, k} = 0$ for each $n$ and hence $\lim_{n \rightarrow \infty} \mathscr{Z}_{n, k} = 0$ trivially. Suppose that $k \geq 2$. Observe that expansion of the term inside the inner sum of the numerator of $\mathscr{Z}_{n, k}$ will result in a linear combination of $1, 1/p, \ldots, 1/p^{j - 1}$. Hence, by multiplying the numerator and denominator of $\mathscr{Z}_{n, k}$ by $p^{\binom{k + 1}{2} - 1}$, one  can rewrite $\mathscr{Z}_{n, k}$ as
\[
\mathscr{Z}_{n, k} =  \frac{\sum\limits_{i = 3}^{k + 1}  \; \sum\limits_{j = 1}^{\binom{i}{2}} \omega_{ij} \binom{n - k - 1}{k + 1 - i} p^{\binom{k + 1}{2} - j}}{\sum\limits_{i = 2}^{k + 1}\; \sum\limits_{j = 1}^{\binom{i}{2}} \xi_{ij} \binom{n - k - 1}{k + 1 - i} p^{\binom{k + 1}{2} - j}}
\]
for some real constants $\{\omega_{ij}\}$ and $\{\xi_{ij}\}$. Since $\binom{n - k - 1}{k + 1 - i} = \Theta( n^{k + 1 - i})$, it follows that to show $\lim_{n \rightarrow \infty} \mathscr{Z}_{n, k} = 0$ one only needs to show that $\lim_{n \rightarrow \infty} \mathscr{Z}^{\prime}_{n, k} = 0$, where
\[
\mathscr{Z}^\prime_{n, k} :=  \frac{\sum\limits_{i = 3}^{k + 1} \; \sum\limits_{j = 1}^{\binom{i}{2}} \tilde{\omega}_{ij} \; n^{k + 1 - i} p^{\binom{k + 1}{2} - j}}{\sum\limits_{i = 2}^{k + 1} \; \sum\limits_{j = 1}^{\binom{i}{2}} \tilde{\xi}_{ij } \; n^{k + 1 - i} p^{\binom{k + 1}{2} - j}}
\]
with $\{\tilde{\omega}_{ij}\}$ and $\{\tilde{\xi}_{ij}\}$ being additional sets of real constants. Since $p = n^\alpha$, the power of $n$ in the summand of numerator as well as denominator of $\mathscr{Z}^\prime_{n, k}$ is of the form
\[
k + 1 - i + \alpha \left[\tbinom{k + 1}{2} - j \right].
\]
Since  $\alpha < 0$, we have
\begin{equation}
\label{eqn:InnerSumMaximum}
\underset{1 \leq j \leq \tbinom{i}{2}}{\arg \max}\left(k + 1 - i + \alpha \left[\tbinom{k + 1}{2} - j\right]\right) = \tbinom{i}{2}.
\end{equation}
Further, the restriction that $\alpha > -1/k$ implies that, for each $i \leq k$,
\begin{equation}
\label{eqn:OuterSumMaximum}
k + 1 - i + \alpha \left[\tbinom{k + 1}{2} - \tbinom{i}{2} \right] \geq k + 1 - (i + 1) + \alpha \left[\tbinom{k + 1}{2} - \tbinom{i + 1}{2} \right].
\end{equation}
From \eqref{eqn:InnerSumMaximum} and \eqref{eqn:OuterSumMaximum}, it follows that the largest power of $n$ in the numerator of $\mathscr{Z}^\prime_{n, k}$ is
\begin{equation}
\label{eqn:NumLargestPower}
k + 1 - 3 + \alpha \left[\tbinom{k + 1}{2} - \tbinom{3}{2} \right],
\end{equation}
while, in the denominator, it is
\begin{equation}
\label{eqn:DenomLargestPower}
k + 1 - 2 + \alpha \left[\tbinom{k + 1}{2} - \tbinom{2}{2} \right].
\end{equation}
Since $k \geq 2$ and hence $\alpha \geq -1/2$, it follows that the term in \eqref{eqn:DenomLargestPower} is larger than that in \eqref{eqn:NumLargestPower}. This shows that $\lim_{n \rightarrow \infty} \mathscr{Z}^\prime_{n, k} = 0$ as desired, and so  completes the proof.
\end{proof}

\begin{lemma}
\label{lem:AsymptoticCovarianceEulerCharacteristic}
Fix $k \geq 1$. Let $p = n^\alpha$, $\alpha \in \left(-\tfrac{1}{k}, -\tfrac{1}{k + 1}\right)$. Then, for any $t_1, t_2 \geq 0$,
\[
\underset{n \rightarrow \infty}{\lim}{\rm{Cov}}[\bar{\chi}_{n}(t_1), \bar{\chi}_{n}(t_2)] = e^{-\lambda |t_1 - t_2|}.
\]
\end{lemma}
\begin{proof}
We need to show that
\[
\lim_{n \rightarrow \infty} \frac{{\rm{Cov}}[\chi_{n}(t_1), \chi_{n}(t_2)]}{\sqrt{{\rm{Var}}[\chi_{n}(t_1)] {\rm{Var}}[\chi_{n}(t_2)]}} = e^{-\lambda|t_1 - t_2|}.
\]
However, since Lemma~\ref{lem:AsymptoticVarianceEulerCharacteristicCliqueCountSame} holds, it suffices to show that
\[
\lim_{n \rightarrow \infty} \frac{{\rm{Cov}}[\chi_{n}(t_1), \chi_{n}(t_2)]}{\sqrt{{\rm{Var}}[f_{n, k}(t_1)]{\rm{Var}}[f_{n, k}(t_2)]}} = e^{-\lambda|t_1 - t_2|}.
\]
But the term inside limit on the left hand side equals
\begin{multline*}
\frac{{\rm{Cov}}[f_{n, k}(t_1), f_{n, k}(t_2)]}{\sqrt{{\rm{Var}}[f_{n, k}(t_1)] {\rm{Var}}[f_{n, k}(t_2)]}} + \frac{{\rm{Cov}}[(-1)^k\chi_{n}(t_1)  - f_{n, k}(t_1), f_{n, k}(t_2)]}{\sqrt{{\rm{Var}}[f_{n, k}(t_1)] {\rm{Var}}[f_{n, k}(t_2)]}} \\ +  \; \frac{{\rm{Cov}}[f_{n, k}(t_1), (-1)^k\chi_{n}(t_2)  - f_{n, k}(t_2)]}{\sqrt{{\rm{Var}}[f_{n, k}(t_1)] {\rm{Var}}[f_{n, k}(t_2)]}} \\ + \; \frac{{\rm{Cov}}[(-1)^k\chi_{n}(t_1)  - f_{n, k}(t_1), (-1)^k\chi_{n}(t_2)  - f_{n, k}(t_2)]}{\sqrt{{\rm{Var}}[f_{n, k}(t_1)]{\rm{Var}}[f_{n, k}(t_2)]}}.
\end{multline*}
Lemma~\ref{lem:AsymptoticCovarianceCliqueCount} shows that the first term converges to $e^{-\lambda|t_1 - t_2|}.$ The remaining terms go to $0$ due to Lemma~\ref{lem:AsymptoticDifferenceCliqueCountEulerCharacterisitc} and the Cauchy-Schwarz inequality and we are done.
\end{proof}

In the above proof, by replacing $f_{n, k}(t_i)$ with $\chi_{n}(t_i)$ and $\chi_{n}(t_i)$ with $\beta_{n, k}(t_i)$ and using Lemmas~\ref{lem:AsymptoticVarianceBettiNumberEulerCharacteristicSame}, \ref{lem:AsymptoticCovarianceEulerCharacteristic}, and \ref{lem:AsymptoticDifferenceBettiNumberEulerCharacterisitc}, appropriately, the following result is easy to prove.

\begin{theorem}
\label{thm:AsymptoticCovarianceBettiNumbers}
Fix $k \geq 1$. Let $p = n^\alpha$, $\alpha \in \left(-\tfrac{1}{k}, -\tfrac{1}{k + 1}\right)$. Then, for any $t_1, t_2 \geq 0$,
\[
\underset{n \rightarrow \infty}{\lim}{\rm{Cov}}[\bar{\beta}_{n, k}(t_1), \bar{\beta}_{n, k}(t_2)] = e^{-\lambda |t_1 - t_2|}.
\]
\end{theorem}

\section{Convergence of Finite Dimensional Distributions}
\label{sec:FiniteDimensionalDistribution}
We now turn to the convergence of the
 finite dimensional distributions of the processes $\bar\beta_{n,k}$, which we establish by  first proving similar results for    $\bar f_{n,k}$ and $\bar\chi_{n,k}$.

For random variables $X, Y$,  write $X \overset{d}{=} Y$ to indicate equivalence in distribution.

\begin{lemma}
\label{lem:AsymptoticMultivariateGaussianCliqueCount}
Fix $k \geq 1$. Let $p = n^\alpha$, $\alpha \in \left(-\tfrac{1}{k}, -\tfrac{1}{k + 1}\right)$. Then, for any $m \in \mathbb{N}$ and any $t_1, \ldots, t_m \geq 0,$ as $n \rightarrow \infty$,
\[
(\bar{f}_{n, k}(t_1), \ldots, \bar{f}_{n, k}(t_m)) \Rightarrow (\mathcal{U}_{\lambda}(t_1), \ldots, \mathcal{U}_{\lambda}(t_m)).
\]
\end{lemma}
\begin{proof}
Fix $m \in \mathbb{N}$, arbitrary $t_1, \ldots, t_m \geq 0$, and arbitrary $\omega_1, \ldots, \omega_m \in \mathbb{R}$. Due to the Cram\'er-Wold theorem \cite[Theorem 29.4]{billingsley2008probability}, it suffices to show that, as $n \rightarrow \infty$,
\begin{equation}
\label{eqn:SufficientResultAsymptoticConvergenceMultivariateGaussian1}
\omega_1 \bar{f}_{n, k}(t_1) + \cdots + \omega_m \bar{f}_{n, k}(t_m) \Rightarrow \omega_1\mathcal{U}_{\lambda}(t_1) + \cdots + \omega_m\mathcal{U}_{\lambda}(t_m).
\end{equation}
But as $\{\mathcal{U}_{\lambda}(t) : t \geq 0\}$ is Gaussian with $\mathbb{E}[\mathcal{U}_{\lambda}(t)] \equiv 0$ and ${\rm{Cov}}[\mathcal{U}_{\lambda}(t_i), \mathcal{U}_{\lambda}(t_j)] = e^{-|t_i - t_j|},$
\[
\frac{\omega_1\mathcal{U}_{\lambda}(t_1) + \cdots + \omega_m\mathcal{U}_{\lambda}(t_m)}{\sqrt{\omega_1^2 + \cdots + \omega_m^2 + 2 \sum_{i < j}\omega_i \omega_j e^{-|t_i - t_j|}}} \overset{d}{=} \mathcal{N}(0,1).
\]
Further, Lemma~\ref{lem:AsymptoticCovarianceCliqueCount} shows that
\begin{equation}
\label{eqn:AsymptoticVarianceCliqueCount}
\underset{n \rightarrow \infty}{\lim} \frac{\sqrt{{\rm{Var}}[\sum_{i = 1}^m \omega_i \bar{f}_{n,k}(t_i)]}}{\sqrt{\omega_1^2 + \cdots + \omega_m^2 + 2 \sum_{i < j}\omega_i \omega_j e^{-|t_i - t_j|}}} = 1.
\end{equation}
Hence, it follows that to prove \eqref{eqn:SufficientResultAsymptoticConvergenceMultivariateGaussian1} we only need show that, as $n \rightarrow \infty$,
\begin{equation}
\label{eqn:SufficientResultAsymptoticConvergenceMultivariateGaussian2}
\mathscr{W}_{n, k } := \frac{\omega_1 \bar{f}_{n, k}(t_1) + \cdots + \omega_m \bar{f}_{n, k}(t_m)}{\sqrt{{\rm{Var}}[\sum_{i = 1}^m \omega_i \bar{f}_{n,k}(t_i)]}} \Rightarrow \mathcal{N}(0,1).
\end{equation}

From \eqref{eqn:ScaledShiftedCliqueCountProcess} and \eqref{eqn:CliqueCount}, we have
\[
\mathscr{W}_{n, k} = \frac{\sum_{A \in \binom{[n]}{k + 1}} \left[ \sum_{i = 1}^{m}  \; \omega_i \bar{1}_{A}(t_i) \right]}{\sqrt{{\rm{Var}}[\sum_{i = 1}^m \omega_i \bar{f}_{n,k}(t_i)]}},
\]
where $\bar{1}_A(t_i) = \left( \frac{1_A(t_i) - \mathbb{E}[1_A(t_i)]}{\sqrt{{\rm{Var}}[f_{n, k}(t_i)]}}\right).$ Indexing the random variable $\left[ \sum_{i = 1}^{m}  \; \omega_i \bar{1}_{A}(t_i) \right]$ with the $\binom{k + 1}{2}$ edges in $A,$ it is easy to see that $\left\{\frac{ \left[ \sum_{i = 1}^{m}  \; \omega_i \bar{1}_{A}(t_i) \right]}{\sqrt{{\rm{Var}}[\sum_{i = 1}^m \omega_i \bar{f}_{n,k}(t_i)]}} : A \in \binom{[n]}{k + 1}\right\} $ is a dissociated set of random variables. For any $A_1 \in \binom{[n]}{k + 1},$ its dependency neighbourhood $\cY(A_1) = \{A_2 \in \binom{[n]}{k + 1} : a_{12} \geq 2\}.$ Here $a_{12}$ denotes the number of vertices common to $A_1$ and $A_2.$ For details, see the discussion above (3.5) in \cite{barbour1989central}.

Let $S_{n, k , m}$ be the cartesian product $\binom{[n]}{k + 1} \times [m].$ For $(A_1 , i )\in S_{n, k, m}$, let $\aleph(A_1, i) = \cY(A_1) \times [m].$ Since $\mathbb{E}[\omega_i \bar{1}_A(t_i)] = 0$ and $\mathbb{E}[\mathscr{W}_{n, k}^2] = 1,$  Theorem~\ref{thm:WassersteinMetricUB} yields that
\begin{multline*}
d_1(\mathscr{W}_{n, k}, \mathcal{N}(0,1))  \leq \frac{\rho \omega^3}{\left({\rm{Var}}[\sum_{i = 1}^m \omega_i \bar{f}_{n,k}(t_i)]\right)^{3/2}}\times \\
\sum_{(A_1, i) \in S_{n, k, m}}
\sum_{
\substack{\text{\tiny $(A_2, j), (A_3, \ell)$} \\
\text{\tiny{$\in \aleph(A_1, i)$}}}} \bigg[ \mathbb{E}\big[|\bar{1}_{A_1}(t_i) \bar{1}_{A_2}(t_j) \bar{1}_{A_3}(t_\ell)|\big]  + \mathbb{E}\big[|\bar{1}_{A_1}(t_i) \bar{1}_{A_2}(t_j)|\big] \; \mathbb{E}\big[|\bar{1}_{A_3}(t_\ell)|\big] \bigg],
\end{multline*}
where $\omega = \max_{i \in [m]} |\omega_i|$.
Since
\begin{multline*}
\mathbb{E}\big[|\bar{1}_{A_1}(t_i) \bar{1}_{A_2}(t_j) \bar{1}_{A_3}(t_\ell)|\big] +  \mathbb{E}\big[|\bar{1}_{A_1}(t_i) \bar{1}_{A_2}(t_j)|\big] \; \mathbb{E}\big[|\bar{1}_{A_3}(t_\ell)|\big] \\
\leq  \dfrac{16 \; \mathbb{E}\big[1_{A_1}(t_i) 1_{A_2}(t_j) 1_{A_3}(t_\ell)\big]}{\sqrt{{\rm{Var}}[f_{n, k}(t_i)] {\rm{Var}}[f_{n, k}(t_j)] {\rm{Var}}[f_{n, k}(t_\ell)]}},
\end{multline*}
\begin{multline*}
d_1(\mathscr{W}_{n, k}, \mathcal{N}(0,1))\\  \leq 16 \rho \omega^3 \frac{\sum\limits_{ (A_1, i) \in S_{n, k, m}}
\; \sum\limits_{
\substack{\text{\tiny $(A_2, j), (A_3, \ell)$} \\
\text{\tiny{$\in \aleph(A_1, i)$}}} }
\mathbb{E}\big[1_{A_1}(t_i) 1_{A_2}(t_j) 1_{A_3}(t_\ell)\big]}{\left({\rm{Var}}[\sum_{i = 1}^m \omega_i \bar{f}_{n,k}(t_i)]\right)^{3/2} {\sqrt{{\rm{Var}}[f_{n, k}(t_i)] {\rm{Var}}[f_{n, k}(t_j)] {\rm{Var}}[f_{n, k}(t_\ell)]}}}.
\end{multline*}
Combining this with \eqref{eqn:AsymptoticVarianceCliqueCount}, and defining
\begin{equation*}
%\label{eqn:SufficientResultAsymptoticConvergenceMultivariateGaussian3}
%\lim_{n \rightarrow \infty}
\mathscr{R}_{n, k}\ :=\  \frac{\sum\limits_{ (A_1, i) \in S_{n, k, m}}
\; \sum\limits_{
\substack{\text{\tiny $(A_2, j), (A_3, \ell)$} \\
\text{\tiny{$\in \aleph(A_1, i)$}}} }
\mathbb{E}\big[1_{A_1}(t_i) 1_{A_2}(t_j) 1_{A_3}(t_\ell)\big]}{\sqrt{{\rm{Var}}[f_{n, k}(t_i)] {\rm{Var}}[f_{n, k}(t_j)] {\rm{Var}}[f_{n, k}(t_\ell)]}},
\end{equation*}
it follows that to establish \eqref{eqn:SufficientResultAsymptoticConvergenceMultivariateGaussian2} we  need only  show that $\lim_{n \rightarrow \infty} \mathscr{R}_{n, k} = 0.$ Fix arbitrary $t \geq 0$ and let $\aleph(A_1)  \equiv \aleph_{n, k}(A_1) := \{A_2 \in \tbinom{[n]}{k + 1} : a_{12} \geq 2\}$ and
\[
\mathscr{R}^\prime_{n, k} := \frac{\sum\limits_{\text{\tiny  $A_1 \in \tbinom{[n]}{k + 1}$}}
\; \sum\limits_{\text{\tiny $A_2, A_3 \in \aleph(A_1)$}} \\
\mathbb{E}\big[1_{A_1}(t) 1_{A_2}(t) 1_{A_3}(t)\big]}{\left({\rm{Var}}[f_{n, k}(t)]\right)^{3/2}}.
\]
In \cite{kahle2013limit}, as part of proof of Claim 2.5 (ii), it was shown that $\lim_{n \rightarrow \infty} \mathscr{R}^\prime_{n, k} = 0$. In the remaining part of this proof, we shall show that
\begin{equation}
\label{eqn:SufficientResultAsymptoticConvergenceMultivariateGaussian4}
\mathscr{R}_{n, k} \leq m^3 \mathscr{R}^\prime_{n, k}.
\end{equation}
This is clearly sufficient to establish $\lim_{n \rightarrow \infty} \mathscr{R}_{n, k} = 0.$

Recall from \eqref{eqn:VarianceCliqueCount} that ${\rm{Var}}[f_{n, k}(t)]$ is independent of $t$. Hence, it follows that the denominators in $\mathscr{R}_{n, k}$ and $\mathscr{R}^\prime_{n, k}$ are identical.  Now using \eqref{eqn:EdgeTransition11} and \eqref{eqn:EdgeProbability} and the fact that $p + (1 - p) e^{-\tau} \leq  1$ for any $\tau \geq 0$, observe that
\begin{eqnarray*}
\mathbb{E}\big[1_{A_1}(t_i) 1_{A_2}(t_j) 1_{A_3}(t_\ell)\big] & \leq &  p^{3\binom{k + 1}{2} - \binom{a_{12}}{2} - \binom{a_{13}}{2} -\binom{a_{23}}{2} + \binom{a_{123}}{2}}\\
& = & \mathbb{E}\big[1_{A_1}(t) 1_{A_2}(t) 1_{A}(t)\big].
\end{eqnarray*}
From this and the definition of $\mathscr{R}_{n, k}$, \eqref{eqn:SufficientResultAsymptoticConvergenceMultivariateGaussian4} follows easily. Desired result thus follows.
\end{proof}

\begin{lemma}
\label{lem:AsymptoticMultivariateGaussianEulerCharacteristic}
Fix $k \geq 1$. Let $p = n^\alpha$, $\alpha \in \left(-\tfrac{1}{k}, -\tfrac{1}{k + 1}\right)$. Then, for any $m \in \mathbb{N}$ and any $t_1, \ldots, t_m \geq 0$, as $n \rightarrow \infty$,
\[
(\bar{\chi}_{n}(t_1), \ldots, \bar{\chi}_{n}(t_m)) \Rightarrow (\mathcal{U}_{\lambda}(t_1), \ldots, \mathcal{U}_{\lambda}(t_m)).
\]
\end{lemma}
\begin{proof}
As in the proof of Lemma~\ref{lem:AsymptoticMultivariateGaussianCliqueCount}, it suffices to show that, as $n \rightarrow \infty$,
\[
\frac{\omega_1 \bar{\chi}_n(t_1) + \cdots + \omega_m \bar{\chi}_n(t_m)}{\sqrt{{\rm{Var}}[\sum_{i = 1}^{m}\omega_i \bar{\chi}_{n}(t_i)]}} \Rightarrow \mathcal{N}(0,1)
\]
for any $\omega_1, \ldots, \omega_m \in \mathbb{R}.$ Since Lemmas~\ref{lem:AsymptoticCovarianceCliqueCount} and \ref{lem:AsymptoticCovarianceEulerCharacteristic} hold, it in fact suffices to show that
\begin{equation}
\label{eqn:MultivariateEC}
\frac{\omega_1 \bar{\chi}_n(t_1) + \cdots + \omega_m \bar{\chi}_n(t_m)}{\sqrt{{\rm{Var}}[\sum_{i = 1}^{m}\omega_i \bar{f}_{n, k}(t_i)]}} \Rightarrow \mathcal{N}(0,1).
\end{equation}
From Lemmas~\ref{lem:AsymptoticDifferenceCliqueCountEulerCharacterisitc} and \ref{lem:AsymptoticVarianceEulerCharacteristicCliqueCountSame}, and the Cauchy-Schwarz inequality, note that, for all $i$,
\begin{eqnarray*}
 {\rm{Var}}[(-1)^k \bar{\chi}_n(t_i) - \bar{f}_{n, k}(t_i)]
%&& \qquad\qquad
&=& {\rm{Var}}\left[(-1)^k \bar{\chi}_n(t_i) - \frac{(-1)^k\chi_n(t_i) - \mathbb{E}[(-1)^k\chi_n(t_i)]}{\sqrt{{\rm{Var}}[f_{n,k}(t_i)]}}  \right. \\
&  & \qquad  \left. + \frac{(-1)^k\chi_n(t_i) - \mathbb{E}[(-1)^k\chi_n(t_i)]} {\sqrt{{\rm{Var}}[f_{n,k}(t_i)]}} - \bar{f}_{n, k}(t_i)\right] \\
& & \leq  \left(\sqrt{{\rm{Var}}[\chi_n(t_i)]}\left|\frac{1}{\sqrt{{\rm{Var}}[\chi_n(t_i)]}} - \frac{1}{\sqrt{{\rm{Var}}[f_{n,k}(t_i)]}}\right|\right. \\
& &\qquad\ + \; \left.\sqrt{\frac{{\rm{Var}}[(-1)^k\chi_n(t_i) - f_{n,k}(t_i)]}{{\rm{Var}}[f_{n,k}(t_i)]}} \right)^2 \\
&&  \to  0,
\end{eqnarray*}
as $n \to \infty$. Using the above estimate and the Cauchy-Schwarz inequality, we find
\beqq
&&{\rm{Var}}\left[\frac{\sum_{i = 1}^{m} \omega_i [(-1)^k \bar{\chi}_n(t_i) - \bar{f}_{n, k}(t_i)]}{\sqrt{{\rm{Var}}[\sum_{i = 1}^{m}\omega_i \bar{f}_{n, k}(t_i)]}}\right]
\\ && \qquad\qquad
 \leq  \left(\sum_{i=1}^m |\omega_i|\sqrt{\frac{{\rm{Var}}[[(-1)^k \bar{\chi}_n(t_i) - \bar{f}_{n, k}(t_i)]]}{{\rm{Var}}[\sum_{i = 1}^{m}\omega_i \bar{f}_{n, k}(t_i)]}}\right)^2
\quad  \to \ \  0,
\eeqq
as $n \to \infty$. From this, it follows that $\left[\frac{\sum_{i = 1}^{m} \omega_i [(-1)^k \bar{\chi}_n(t_i) - \bar{f}_{n, k}(t_i)]}{\sqrt{{\rm{Var}}[\sum_{i = 1}^{m}\omega_i \bar{f}_{n, k}(t_i)]}}\right]$ converges to $0$ in probability. Since Lemma~\ref{lem:AsymptoticMultivariateGaussianCliqueCount} holds and
\[
\frac{\sum_{i = 1}^{m}\omega_i \bar{\chi}_n(t_i)}{\sqrt{{\rm{Var}}[\sum_{i = 1}^{m}\omega_i \bar{f}_{n, k}(t_i)]}} = \frac{\sum_{i = 1}^{m} \omega_i [(-1)^k \bar{\chi}_n(t_i) - \bar{f}_{n, k}(t_i)]}{\sqrt{{\rm{Var}}[\sum_{i = 1}^{m}\omega_i \bar{f}_{n, k}(t_i)]}} + \frac{\sum_{i = 1}^{m}\omega_i \bar{f}_{n, k}(t_i)}{\sqrt{{\rm{Var}}[\sum_{i = 1}^{m}\omega_i \bar{f}_{n, k}(t_i)]}},
\]
\eqref{eqn:MultivariateEC} follows via Slutsky's theorem \cite[Chapter 6, Theorem 6.5]{gut2009intermediate} and so does the claim.
\end{proof}

\begin{theorem}
\label{thm:AsymptoticMultivariateGaussianBettiNumber}
Fix $k \geq 1$. Let $p = n^\alpha$, $\alpha \in \left(-\tfrac{1}{k}, -\tfrac{1}{k + 1}\right)$. Then, for any $m \in \mathbb{N}$ and any $t_1, \ldots, t_m$, as $n \rightarrow \infty$,
\[
(\bar{\beta}_{n, k}(t_1), \ldots, \bar{\beta}_{n, k}(t_m)) \Rightarrow (\mathcal{U}_{\lambda}(t_1), \ldots, \mathcal{U}_{\lambda}(t_m)).
\]
\end{theorem}
\begin{proof}
The arguments are similar to those used in the proof of Lemma~\ref{lem:AsymptoticMultivariateGaussianEulerCharacteristic}. Firstly, using  Lemma~\ref{lem:AsymptoticDifferenceBettiNumberEulerCharacterisitc}, it follows that for any $\omega_1, \ldots, \omega_m \in \mathbb{R}$,
\[
\lim_{n \rightarrow \infty}{\rm{Var}}\left[\frac{\sum_{i = 1}^{m} \omega_i [(-1)^k \bar{\chi}_n(t_i) - \bar{\beta}_{n, k}(t_i)]}{\sqrt{{\rm{Var}}[\sum_{i = 1}^{m}\omega_i \bar{\chi}_{n, k}(t_i)]}}\right] = 0.
\]
Then using Lemma~\ref{lem:AsymptoticCovarianceEulerCharacteristic} and Theorem~\ref{thm:AsymptoticCovarianceBettiNumbers}, the desired result follows.
\end{proof}

\section{Tightness}
\label{sec:Tightness}

In this section, we show that, for each $k$,  the sequences $\{\hat\beta_{n,k} \: n \geq 1\}$ are tight. By Theorem~\ref{thm:SufficientConditionsConvergenceOrnsteinUhlenbeckProcess}, it suffices to establish the two conditions \ref{cond: InitialDrift} and \ref{cond: Tightness} for these sequences.

\begin{lemma}
\label{lem:InitialDriftBettiNumbers}
Fix $k \geq 1$. Let $p = n^\alpha$, $\alpha \in \left(-\tfrac{1}{k}, -\tfrac{1}{k + 1}\right)$. Then, for the sequence $\{ \hat\beta_{n,k}\:n \geq 1\}$, condition \ref{cond: InitialDrift} holds  with $\Upsilon = 2$, i.e.,
\[
\lim_{\delta \rightarrow 0} \limsup_{n \rightarrow \infty} \mathbb{E}[\bar{\beta}_{n, k}(\delta) - \bar{\beta}_{n, k}(0)]^2 = 0.
\]
\end{lemma}
\begin{proof}
From Theorem~\ref{thm:AsymptoticCovarianceBettiNumbers} and the fact that
\[
\mathbb{E}[\bar{\beta}_{n, k}(\delta) - \bar{\beta}_{n, k}(0)]^2 = 2 - 2 {\rm{Cov}}[\bar{\beta}_{n, k}(\delta),\bar{\beta}_{n, k}(0)],
\]
\[
\lim_{n \rightarrow \infty} \mathbb{E}[\bar{\beta}_{n, k}(\delta) - \bar{\beta}_{n, k}(0)]^2 = 2 - 2 e^{-\delta}.
\]
and the result follows easily.
\end{proof}

Arguing as above, it follows that \ref{cond: InitialDrift} is also satisfied for the sequence of $\{\hat f_{n,k} : n \geq 1\}$ \! and $\{\hat\chi_{n,k}\: n \geq 1\}$. We now aim to show that \ref{cond: Tightness} holds true for $\{\text{\BetaBar} : n \geq 1\}$. Our approach is to first establish this result for $\text{\CliqueBar} $, then for $\text{\EulerBar}$, and
finally for $\text{\BetaBar} $.

\begin{lemma}
\label{lem:TightnessCliqueCount}
Fix $k \geq 1$. Let $p = n^\alpha$, $\alpha \in \left(-\tfrac{1}{k}, -\tfrac{1}{k + 1}\right)$. Then, for the sequence $\{\text{\CliqueBar} : n \geq 1\}$, condition \ref{cond: Tightness} holds with $\Upsilon_1 = \Upsilon_2 = 2$. That is, for any $T > 0$, there exists $K_f > 0$ such that, for all $n \geq 1$, $0 \leq t \leq T + 1$ and $0 \leq h \leq t$,
\[
\mathbb{E}\left[\bar{f}_{n, k}(t + h) - \bar{f}_{n, k}(t)\right]^2 \left[\bar{f}_{n, k}(t) - \bar{f}_{n, k}(t - h)\right]^2 \leq K_f h^2.
\]
\end{lemma}

This follows from the next result and hence we prove only that.

\begin{lemma}
\label{lem:TightnessEulerCharacteristic}
Fix $k \geq 1$. Let $p = n^\alpha$, $\alpha \in \left(-\tfrac{1}{k}, -\tfrac{1}{k + 1}\right)$. Then, for the sequence $\{\text{\EulerBar} : n \geq 1\}$, condition \ref{cond: Tightness} holds with $\Upsilon_1 = \Upsilon_2 = 2.$ That is, for any $T > 0$, there exists $K_\chi > 0$ such that, for all $n \geq 1$, $0 \leq t \leq T + 1$ and $0 \leq h \leq t$,
\[
\mathbb{E}\left[\bar{\chi}_{n, k}(t + h) - \bar{\chi}_{n, k}(t)\right]^2 \left[\bar{\chi}_{n, k}(t) - \bar{\chi}_{n, k}(t - h)\right]^2 \leq K_\chi h^2.
\]
\end{lemma}

Before turning to the proof of Lemma \ref{lem:TightnessEulerCharacteristic} we need some additional notation and preliminary lemmas. Fix arbitrary $n, k \geq 1$ and let $p$ be as in Lemma~\ref{lem:TightnessEulerCharacteristic}. Also fix $i$ and $j$ such that $0 \leq i, j \leq n - 1$ and let
\begin{equation}
\label{defn:CrossMomentij}
\xi_{ij}(h) := \mathbb{E}[f_{n, i}(2h) - f_{n, i}(h)]^2[f_{n, j}(h) - f_{n, j}(0)]^2.
\end{equation}
For $\bar{A} \equiv (A_1, A_2, A_3, A_4) \in \tbinom{[n]}{i + 1}^2 \times \tbinom{[n]}{j + 1}^2$, let $a_q$ be the number of vertices in $A_q$, $a_{qr}$ be the number of vertices common to $A_q$ and $A_r$, and so on. Note that inequalities such as $a_{1234} \leq a_{qrs} \leq a_{qr} \leq a_q$ for any $q,r,s \in \{1, \ldots, 4\}$ hold trivially. Let
\[
\tau(\bar{A}) = (a_1, \ldots, a_4, a_{12}, \ldots, a_{34}, a_{123}, \ldots, a_{234}, a_{1234}),
\]
\begin{equation}
\label{eqn:defn_Ver}
\text{\ver{A}} = \sum_{q = 1}^4 a_q - \sum_{1 \leq q < r \leq 4} a_{qr} + \sum_{1\leq q < r < s \leq 4} a_{qrs} - a_{1234},
\end{equation}
\begin{equation}
\label{eqn:defn_Pair}
\text{\pair{A}} = \sum_{q = 1}^4 \binom{a_q}{2} - \sum_{1 \leq q < r \leq 4} \binom{a_{qr}}{2} + \sum_{1\leq q < r < s \leq 4} \binom{a_{qrs}}{2} - \binom{a_{1234}}{2},
\end{equation}
and
\begin{multline}
\label{eqn:defn_gEulerCharacteristic}
g(h; \bar{A}) := \left[1_{A_1}(2h) - 1_{A_1}(h)\right]\left[1_{A_2}(2h) - 1_{A_2}(h)\right]\\ \times \left[1_{A_3}(h) - 1_{A_3}(0)\right]\left[1_{A_4}(h) - 1_{A_4}(0)\right].
\end{multline}
Here $\tau(\bar{A})$ denotes the intersection type of $\bar{A}$, while \ver{A} and \pair{A} denote respectively the number of vertices and maximum possible edges in $A_1, \ldots, A_4$ with common vertices and edges counted only once. Terms of the form $g(h; \bar{A})$ appear in the expansion of $\xi_{ij}(h)$ and hence will be useful later.

For $\bar{A},\bar{B} \in \tbinom{[n]}{i + 1}^2 \times \tbinom{[n]}{j + 1}^2$, we write $\bar{A} \sim \bar{B}$ if there exists a permutation $\pi$ of the sets in $\bar{B}$ such that $\tau(\bar{A}) = \tau(\pi(\bar{B}))$. A priori, it may appear that the intersection type of all $24$ permutations of the sets in $\bar{B}$ need to be compared with $\tau(\bar{A})$ before concluding $\bar{A} \sim \bar{B}$ or not. But this is true only when $i = j.$ When $i \neq j$, many of the permutations need not be checked. For example, the permutation that interchanges the first and third set can be ignored. Clearly, $\sim$ is an equivalence relation. Let $\Gamma_{ij} := \{[\bar{A}]\}$ denote the quotient of $\tbinom{[n]}{i + 1}^2 \times \tbinom{[n]}{j + 1}^2$ under $\sim$, where $[\bar{A}]$ denotes the equivalence class of $\bar{A}$. Since each $a_{qr}, a_{qrs}$, and $a_{1234}$ ($11$ variables in total) is a number between $0$ and $\max\{i+ 1,j + 1\} \leq (i +j + 1)$, the cardinality of $\Gamma_{ij}$ satisfies
\begin{equation}
\label{eqn:CardinalityGammaij}
|\Gamma_{ij}| \leq (i + j + 1)^{11}.
\end{equation}

We shall say $\bar{A} \in \tbinom{[n]}{i + 1}^2 \times \tbinom{[n]}{j + 1}^2$ has an {\it independent set} if there exists $q \in \{1, 2, 3, 4\}$ such that $a_{qr} \leq 1$ for all $r\neq q$. That is, there exists a special set among $A_1, \ldots, A_4$ which shares at most one vertex with the remaining three sets. Clearly, the indicator associated with this special set is independent of the indicator associated with the other three sets. Based on this description,  let
\begin{equation}
\label{eqn:defn_IndependentSet}
\mathscr{S}_{ij} := \left\{[\bar{A}] \in \Gamma_{ij} : \exists q \in \{1, 2, 3, 4\} \text{ such that } \forall r \neq q, a_{qr} \leq 1\right\}.
\end{equation}

\begin{lemma}
\label{claim:Exp_gVal}
Fix arbitrary $n, k \geq 1$,  and let $p$ be as in Lemma~\ref{lem:TightnessEulerCharacteristic}. Also fix $i$ and $j$ such that $0 \leq i, j \leq n - 1$. Fix $\bar{A} \in \tbinom{[n]}{i + 1}^2 \times \tbinom{[n]}{j + 1}^2$.
\begin{enumerate}
\item[(i)] If $[\bar{A}] \in \mathscr{S}_{ij}$, then $\mathbb{E}[g(h; \bar{A})] \equiv 0$.

\item[(ii)] If $[\bar{A}] \in \Gamma_{ij} \backslash \mathscr{S}_{ij}$, then there exists some universal constant $\gamma \geq 0$ (independent of $\bar{A}$, $i,j,k$, and $n$) such that,
for all $0 \leq h \leq 1$,
\[
|\mathbb{E}[g(h; \bar{A})]| \leq \gamma (i + j + 1)^4 p^{\text{pair}(\bar{A})} h^2.
\]
\end{enumerate}
\end{lemma}
\begin{proof}
The first claim is straightforward and follows from the stationarity of the dynamic \ER graph. So we discuss only the second one.

Fix $\bar{A} \in \tbinom{[n]}{i + 1}^2 \times \tbinom{[n]}{j + 1}^2$ with $[\bar{A}] \in \Gamma_{ij} \backslash \mathscr{S}_{ij}$. It is tedious but not difficult to see that $g(h; \bar{A})$ satisfies \eqref{eqn:gExpansion}, cf.\ Appendix \ref{sec:app}. Hence using \eqref{eqn:EdgeTransition11} and \eqref{eqn:EdgeProbability}, we have
\begin{equation}
\label{eqn:Expg}
\mathbb{E}[g(h; \bar{A})] = p^{\text{pair}(\bar{A})} \Phi(h; \bar{A}),
\end{equation}
where $\Phi(h; \bar{A})$ is as in \eqref{eqn:PhiVal}. Note that $\Phi(h; \bar{A})$ has the form
\begin{equation}
\label{eqn:PhiSum}
\Phi(h; \bar{A}) = \sum_{\ell = 1}^{16} \phi_\ell(h; \bar{A}),
\end{equation}
where, for each $\ell$,
\begin{equation}
\label{eqn:Phi_l}
\phi_{\ell}(h; \bar{A}) = \pm ((1 - p) e^{-h} + p)^{c_1(\ell)} ((1 - p) e^{-2h} + p)^{c_2(\ell)}
\end{equation}
with
\begin{equation}
\label{eqn:gcoeffcientUpperBound}
0 \leq c_1(\ell), c_2(\ell) \leq \sum_{1 \leq q < r \leq 4} \tbinom{a_{qr}}{2} + \tbinom{a_{1234}}{2} \leq 7(i + j + 1)^2.
\end{equation}
By analysing \eqref{eqn:PhiVal}, it is not difficult to see that
\[
\left.\Phi(h; \bar{A})\right|_{h = 0} = 0, \text{ and }  \left.\frac{\partial \Phi(h; \bar{A})}{\partial h}\right|_{h = 0} = 0.
\]
Because of the above two facts, expanding $\Phi(h; \bar{A})$ using the Lagrangian form of Taylor series shows that, for each $0 \leq h \leq 1$, there exists $c \in [0,h]$ such that
\begin{equation}
\label{eqn:UpperBoundPhi}
\Phi(h; \bar{A}) = \frac{1}{2}h^2 \left.\frac{\partial^2 \Phi(h; \bar{A})}{\partial h^2}\right|_{h = c}.
\end{equation}
Now using \eqref{eqn:PhiSum}, \eqref{eqn:Phi_l}, \eqref{eqn:gcoeffcientUpperBound}, and the fact that both $((1 - p) e^{-h} + p)$ and $((1 - p)e^{-2h} + p)$ are bounded from above by $1$ for $h \geq 0$, it is not difficult to see that there exists some universal constant $\gamma_1 \geq 0$ (independent of $\bar{A}, i, j, k$, and $n$) such that
\[
\max_{1 \leq \ell \leq 16} \sup_{h \geq 0} \left|\frac{\partial^2 \phi_\ell(h; \bar{A})}{\partial h^2}\right| \leq \gamma_1 (i + j + 1)^4.
\]
Combining this with \eqref{eqn:PhiSum} and \eqref{eqn:UpperBoundPhi}, it follows that $|\Phi(h; \bar{A})| \leq 8\gamma_1  (i + j + 1)^4 h^2.$ Using this inequality in \eqref{eqn:Expg}, the  result follows.
\end{proof}

\begin{lemma}
\label{claim: VerPairBound}
Fix arbitrary $n, k \geq 1$, and let $p$ be as in Lemma~\ref{lem:TightnessEulerCharacteristic}. Also fix $i$ and $j$ such that $0 \leq i, j \leq n - 1$. Fix $\bar{A} \in \tbinom{[n]}{i + 1}^2 \times \tbinom{[n]}{j + 1}^2$.
\begin{enumerate}
\item[(i)]  If $[\bar{A}] \in \Gamma_{ij} \backslash \mathscr{S}_{ij}$, then
\[
\frac{n^\text{\ver{A}} p^\text{\pair{A}}}{n^{4k} p^{4\tbinom{k + 1}{2} - 2}} \leq 1.
\]
\item[(ii)]   If $[\bar{A}] \in \Gamma_{ij} \backslash \mathscr{S}_{ij}$ and $(i + j) \geq 16k + 15$, then
\[
\frac{n^\text{\ver{A}} p^\text{\pair{A}}}{n^{4k} p^{4\tbinom{k + 1}{2} - 2}} \leq \frac{1}{n^{2k + 2(i + j - 16k - 15)} }.
\]
\end{enumerate}
\end{lemma}
\begin{proof}
From \eqref{eqn:defn_IndependentSet}, as $\bar{A} \in \Gamma_{ij} \backslash \mathscr{S}_{ij}$, one of the below cases must hold.

\textbf{Case A:} Either $a_{12}, a_{34} \geq 2$, or $a_{13}, a_{24} \geq 2$, or $a_{14}, a_{23} \geq 2$.

\textbf{Case B:} There exists $q \in \{1, \ldots, 4\}$ such that $a_{qr} \geq 2$ for all $r \neq q$.

In both cases, using essentially the same arguments as those used to obtain (8) in \cite{kahle2014erratum}, with the differences noted below, we have
\[
n^{\text{\ver{A}}} p^{\text{\pair{A}}} \leq n^{4k} p^{4 \tbinom{k + 1}{2} - 2}.
\]
This proves the first claim of the lemma  modulo clearing up the two main differences between the arguments needed here and  those used in \cite{kahle2014erratum}. The first relates to the fact that
\cite{kahle2014erratum}  dealt with the intersection of three sets while here we need to deal with four sets. In both cases, however,  independent sets are absent, i.e., each set has at least two vertices in common with one of the remaining sets.

Secondly, in \cite{kahle2014erratum}, an upper bound for $n^{\text{\ver{A}}} p^{\text{\pair{A}}}$, with $\text{ver}(\bar{A})$ and $\text{pair}(\bar{A})$ appropriately defined, was obtained by sequentially dealing with the number of vertices in the third set, then the second set, and so on. Here we have to repeat the same idea by first dealing with the number of vertices in the fourth set, then third, etc.

Now consider the second claim of the lemma.  Again the conditions of Case \textbf{A} and Case \textbf{B} defined above must hold. Hence, from \eqref{eqn:defn_Ver}, \eqref{eqn:defn_Pair}, and \eqref{eqn:defn_IndependentSet}, we have $\text{\ver{A}} \leq 2i + 2j$ and
\[
\text{\pair{A}} \geq \max\left\{ \tbinom{i + 1}{2}, \tbinom{j + 1}{2}\right\}.
\]
Using these and fact that $\alpha \in \left(-\tfrac{1}{k}, -\tfrac{1}{k + 1}\right)$, we have
\[
\text{\ver{A}} + \alpha  \text{\pair{A}} \leq  2(i + j) + \alpha \max \left\{ \tbinom{i + 1}{2}, \tbinom{j + 1}{2}\right\}.
\]
Since $\max \left\{ \tbinom{i + 1}{2}, \tbinom{j + 1}{2}\right\} \geq \tbinom{i + j + 1}{2}/4$, it follows that
\[
\text{\ver{A}} + \alpha  \text{\pair{A}} \leq 2(i + j) + \alpha \tbinom{i + j + 1}{2}/4.
\]
Consequently, to prove the desired result, it suffices to show that for $i + j \geq 16k + 15$,
\begin{equation}
\label{eqn:desResult1}
\frac{n^{2(i + j) + \alpha \tbinom{i + j + 1}{2}/4}}{n^{4k} p^{4\tbinom{k + 1}{2} - 2}} \leq \frac{1}{n^{2k + 2(i + j - 16k - 15)} }.
\end{equation}

Now observe that if $i + j = 16k + 15$, then
\[
\frac{n^{2(i + j)+ \alpha \tbinom{i + j + 1}{2}/4}}{n^{4k + \alpha (4\tbinom{k + 1}{2} - 2)}} \leq \frac{1}{n^{2k}}.
\]
Suppose that for $i^\prime$ and $j^\prime$ with $(i^\prime + j^\prime) \geq 16k + 15$, the desired result holds. Now consider $i$ and $j$ satisfying $(i + j) = (i^\prime + j^\prime) + 1$. Since $(i^\prime + j^\prime) \geq 16k + 15,$
\[
2(i + j) - 2(i^\prime + j^\prime) + \alpha \left[\tbinom{i + j + 1}{2}/4 - \tbinom{i^\prime + j^\prime + 1}{2}/4\right]  =  2 + \alpha \left(i^\prime + j^\prime + 1\right)/4 \leq  -2.
\]
By induction, \eqref{eqn:desResult1} follows and so does the claim.
\end{proof}

\begin{lemma}
\label{claim:UpperBoundXiij}
Fix arbitrary $n, k \geq 1$,  and let $p$ be as in Lemma~\ref{lem:TightnessEulerCharacteristic}. Also fix $i$ and $j$ such that $0 \leq i, j \leq n - 1$. Let $\xi_{ij}(h)$ be as in \eqref{defn:CrossMomentij} and $\gamma$ as in Lemma~\ref{claim:Exp_gVal}.
\begin{enumerate}
\item[(i)] If $(i + j) < 16k + 15$, then
\[
\frac{\xi_{ij}(h)}{n^{4k} p^{4\tbinom{k + 1}{2} - 2}} \leq \gamma \; (i + j + 1)^{15} h^2.
\]

\item[(ii)] If $(i + j) \geq 16k + 15$, then
\[
\frac{\xi_{ij}(h)}{n^{4k} p^{4\tbinom{k + 1}{2} - 2}} \leq \gamma \;  \frac{(i + j + 1)^{15}}{n^{2k + 2 (i + j - 16k - 15)}} h^2.
\]
\end{enumerate}
\end{lemma}
\begin{proof}
From \eqref{defn:CrossMomentij} and \eqref{eqn:defn_gEulerCharacteristic}, it is easy to see that
\[
\xi_{ij}(h) = \sum_{\bar{A} \in \tbinom{i + 1}{2}^2 \times \tbinom{j + 1}{2}^2} \mathbb{E}[g(h; \bar{A})].
\]
Collecting terms based on their equivalence classes under $\sim$, it follows that
\[
\xi_{ij}(h) = \sum_{[\bar{B}] \in \Gamma_{ij}} \;  \sum_{\bar{A} \in \tbinom{i + 1}{2}^2 \times \tbinom{j + 1}{2}^2: \bar{A} \sim \bar{B}} \mathbb{E}[g(h; \bar{A})].
\]
Applying Lemma~\ref{claim:Exp_gVal}  gives
\[
\xi_{ij}(h) \leq \gamma \; (i + j + 1)^4 h^2 \sum_{[\bar{B}] \in \Gamma_{ij} \backslash \mathscr{S}_{ij}} \;  \sum_{\bar{A} \in \tbinom{i + 1}{2}^2 \times \tbinom{j + 1}{2}^2: \bar{A} \sim \bar{B}} p^{\text{\pair{A}}}.
\]
Now note from \eqref{eqn:defn_Ver} and \eqref{eqn:defn_Pair} that,  if $\bar{A} \sim \bar{B}$, then \ver{A} $= $ \ver{B} and \pair{A} $ = $ \pair{B}. Further, the cardinality of the set $\{\bar{A} \in \tbinom{i + 1}{2}^2 \times \tbinom{j + 1}{2}^2: \bar{A} \sim \bar{B}\}$ is  bounded above by $n^{\text{\ver{B}}}$. From these observations, it follows that
\[
\xi_{ij}(h) \leq \gamma \;  (i + j + 1)^4 h^2 \sum_{[\bar{B}] \in \Gamma_{ij} \backslash \mathscr{S}_{ij}} n^{\text{\ver{B}}} p^{\text{\pair{B}}}.
\]
Using \eqref{eqn:CardinalityGammaij} and Lemma~\ref{claim: VerPairBound}, both the desired statements are now easy to see.
\end{proof}

\begin{proof}[Proof of Lemma~\ref{lem:TightnessEulerCharacteristic}]
Since $\{G(n, p, t): t \geq 0\}$ and hence $\{\bar{\chi}_{n}(t) : t \geq 0\}$ are stationary, to prove the desired result, it suffices to show that there exists $K_{\chi} > 0$ such that
\begin{equation}
\label{eqn:desResultEulerTightness1}
\mathbb{E}\left[\bar{\chi}_{n}(2h) - \bar{\chi}_{n}(h)\right]^2 \left[\bar{\chi}_{n}(h) - \bar{\chi}_{n}(0)\right]^2 \leq K_{\chi}h^2
\end{equation}
for $0 \leq h \leq 1$ and $n \geq 1$. From Lemmas~\ref{lem:VarianceCliqueCount} and \ref{lem:AsymptoticVarianceEulerCharacteristicCliqueCountSame}, ${\rm{Var}}[\chi_{n}(t)] = \Theta(n^{2k} p^{2\tbinom{k + 1}{2} - 1})$. Hence, to prove \eqref{eqn:desResultEulerTightness1}, it suffices to show that there exists $K_{\chi} > 0$ such that
\[
\Omega_{n, k}(h) := \frac{\mathbb{E}\left[\chi_{n}(2h) - \chi_{n}(h)\right]^2 \left[\chi_{n}(h) - \chi_{n}(0)\right]^2}{n^{4k} p^{4\tbinom{k + 1}{2} - 2}} \leq K_{\chi}h^2.
\]

Using \eqref{eqn:defnEulerCharacteristicCliqueCount} and the triangle inequality, we have
\[
\sqrt{\Omega_{n, k}(h)} \leq \sum_{0 \leq i,j \leq n - 1} \sqrt{\frac{\xi_{i, j}(h)}{n^{4k} p^{4\tbinom{k + 1}{2} - 2}}},
\]
where $\xi_{ij}(h)$ is as in \eqref{defn:CrossMomentij}. Collecting terms based on the sum $(i + j)$, we have
\[
\sqrt{\Omega_{n, k}(h)} \leq \sum_{0 \leq \ell \leq 2(n - 1)} \sum_{(i + j) = \ell}  \sqrt{\frac{\xi_{i, j}(h)}{n^{4k} p^{4\tbinom{k + 1}{2} - 2}}}.
\]
This implies that
\[
\sqrt{\Omega_{n, k}(h)} \leq \sum_{0 \leq \ell < \infty} \sum_{(i + j) = \ell}  \sqrt{\frac{\xi_{i, j}(h)}{n^{4k} p^{4\tbinom{k + 1}{2} - 2}}}.
\]
From this it follows that $\sqrt{\Omega_{n, k}(h)} \leq \text{Term}_1 + \text{Term}_2,$
where
\[
\text{Term}_1 := \sum_{0 \leq \ell < 16k + 15} \; \sum_{(i + j) = \ell}  \sqrt{\frac{\xi_{i, j}(h)}{n^{4k} p^{4\tbinom{k + 1}{2} - 2}}}
\]
and
\[
\text{Term}_2 := \sum_{16k + 15 \leq \ell < \infty} \; \sum_{(i + j) = \ell}  \sqrt{\frac{\xi_{i, j}(h)}{n^{4k} p^{4\tbinom{k + 1}{2} - 2}}}.
\]
As $\{(i, j): i,j \geq 0, i + j = \ell\}$ has $\ell + 1$ elements, using  (i) of Lemma~\ref{claim:UpperBoundXiij}, we have $\text{Term}_1 \leq \sqrt{\gamma \; h^2} K_1,$ where $K_1 := \sum_{0 \leq \ell < 16k + 15} (l  + 1)^{15/2  + 1}.$
Note that $K_1$ is a constant independent of $n$ and $0 \leq h \leq 1$. Similarly, using (ii)   of Lemma~\ref{claim:UpperBoundXiij}, we obtain $\text{Term}_2 \leq  \sqrt{\gamma \; h^2} K_2(n),$ where
\[
K_2(n) := \sum_{16k + 15 \leq \ell < \infty} \frac{(\ell + 1)^9}{n^{k  + (\ell - 16k - 15)}}.
\]
Clearly $K_2(n)$ is finite for each $n \geq 2$ and is monotonically decreasing. Consequently, if we let $ K_\chi := \gamma (K_1 +  K_2(2))^2$, then the desired result follows.\end{proof}

\begin{theorem}
\label{thm:TightnessBettiNumber}
Fix $k \geq 1$. Let $p = n^\alpha$, $\alpha \in \left(-\tfrac{1}{k}, -\tfrac{1}{k + 1}\right)$. Then, for the sequence $\{\text{\BetaBar} : n \geq 1\}$, condition \ref{cond: Tightness} holds with $\Upsilon_1 = \Upsilon_2 = 2$, i.e., for any $T > 0$, there exists $K_{\beta} > 0$ such that, for all $n \geq 1$, $0 \leq t \leq T + 1,$ and $0 \leq h \leq t$
\[
\mathbb{E}\left[\bar{\beta}_{n, k}(t + h) - \bar{\beta}_{n, k}(t)\right]^2 \left[\bar{\beta}_{n, k}(t) - \bar{\beta}_{n, k}(t - h)\right]^2 \leq K_\beta h^2.
\]
\end{theorem}
\begin{proof}
From Lemmas~ \ref{lem:AsymptoticVarianceEulerCharacteristicCliqueCountSame}, \ref{lem:AsymptoticVarianceBettiNumberEulerCharacteristicSame}, and \ref{lem:VarianceCliqueCount}, we have ${\rm{Var}}[\beta_{n, k}(t)] = \Theta(n^{2k}p^{2\tbinom{k + 1}{2} - 1}).$ Hence, as discussed in Lemma~\ref{lem:TightnessEulerCharacteristic}, to prove the desired result it suffices to show that there exists $K_\beta > 0$ such that
\[
\Omega_{n, k}(h) := \frac{\mathbb{E}\left[\beta_{n, k}(2h) - \beta_{n, k}(h)\right]^2 \left[\beta_{n, k}(h) - \beta_{n, k}(0)\right]^2}{n^{4k}p^{4\tbinom{k + 1}{2} - 2}} \leq K_\beta h^2
\]
for all $n \geq 1$ and $0 \leq h \leq 1$.

Now fix an arbitrary $h \in [0,1]$ and consider the event
\begin{multline}
\label{eqn:EulerEqualBettiEvent}
E = \{(-1)^k\chi_{n}(0) = \beta_{n, k}(0)\} \cap \{(-1)^k\chi_{n}(h) = \beta_{n, k}(h)\} \\ \cap \{(-1)^k\chi_{n}(2h) = \beta_{n, k}(2h)\}.
\end{multline}
Then, observe that
\begin{equation}
\label{eqn:BettiSplit}
\mathbb{E}[\beta_{n, k}(2h) - \beta_{n, k}(h)]^2[\beta_{n, k}(h) - \beta_{n, k}(0)]^2 = \text{Term}_1 + \text{Term}_2,
\end{equation}
where
\begin{equation}
\label{eqn:Term1}
\text{Term}_1 =  \mathbb{E}[\beta_{n, k}(2h) - \beta_{n, k}(h)]^2[\beta_{n, k}(h) - \beta_{n, k}(0)]^2 1_{E}
\end{equation}
and
\begin{equation}
\label{eqn:Term2}
\text{Term}_2 =  \mathbb{E}[\beta_{n, k}(2h) - \beta_{n, k}(h)]^2[\beta_{n, k}(h) - \beta_{n, k}(0)]^2 1_{E^c}.
\end{equation}
Clearly,
\begin{equation}
\label{eqn:Term1Exp1}
\text{Term}_1 = \mathbb{E}[\chi_{n}(2h) - \chi_{n}(h)]^2[\chi_{n}(h) - \chi_{n}(0)]^2 1_{E}
\end{equation}
and hence, using Lemma~\ref{lem:TightnessEulerCharacteristic}, it follows that
\begin{equation}
\label{eqn:Term1Exp3}
\frac{\text{Term}_1}{n^{4k} p^{4\tbinom{k + 1}{2} - 2}}\leq K_\chi h^2.
\end{equation}

To obtain a bound on Term$_2$, we consider an alternate but equivalent description of the dynamic \ER graph. Specifically, to each edge $e$, independently associate two independent sequences $T^e := \{ T^e_i \}_{i \geq 1}$ and $I^e:= \{I^e_i\}_{i \geq 0}$, where the  $T^e$ are arrival times of a Poisson process with parameter $\lambda$ and
the  $I^e$ are i.i.d.\ Bernoulli random variables which take the `on' state with probability $p$ and `off' state with probability $1 - p$. Let $T^e_0 = 0$. If we define the state of the edge $e$ at time $t$ as
\[
e(t) := \sum_{i \geq 0} 1_{\{T^e_i \leq t < T_{i+1}^e\}}I_i^e,
\]
then it follows that the behaviour of edge $e$ is that of an edge in the dynamic \ER graph. Firstly,  the initial configuration ${e}(0) = I_0^{e}$ a.s. and so $\Pr\{e(0) = \text{on}\} = p,$ as required. Fix $t_1 < t_2$. Let $\# _T$ be the cardinality of $\{i : T_i^e \in (t_1, t_2]\}$ and, if $\#_T > 0$, let $i_{\text{last}} := \arg \max\{i : T_i^e \in (t_1, t_2]\}$. Then
\begin{eqnarray*}
\Pr\{e(t_2) = \text{on} | e(t_1) = \text{on}\} & = & \Pr\{\#_T = 0 \} + \sum_{\ell > 0} \Pr\{\#_T = \ell, T_{i_{\text{last}}}^e = \text{on} \}\\
& = & e^{-\lambda(t_2 - t_1)} + \sum_{\ell > 0} e^{-\lambda (t_2 - t_1)} \frac{[\lambda (t_2 - t_1)]^\ell}{\ell !} p,
\end{eqnarray*}
where the last equality follows due to independence of $T^e$ and $I^e$. From this, it is easy to see that \eqref{eqn:EdgeTransition11} holds. Similarly one can check that \eqref{eqn:EdgeTransition00} also holds. This verifies the equivalence of the two descriptions of the dynamic \ER graph.

Let $S_{0, h} := \sum_{e} \sum_{i \geq 1} 1_{\{T^{e}_i \leq h\}}$ denote the sum of arrivals that occurred across each edge in time $(0,h]$. Let $\tau_1, \tau_2, \ldots$, with $\tau_i \leq \tau_{i + 1}$, denote the sequence of arrival times in $(0, h]$ at which these $S_{0,h}$ arrivals occurred. Note that $\tau_i$ and $\tau_{i + 1}$ could correspond to arrivals along different edges. Separately, let $\tau_0 = 0$. Let $\mathcal{P}_0$ denote the event that no arrival occurs at time $0$, i.e., for all $i \geq 1$, $\tau_i > 0$. Then,
\[
|\beta_{n, k}(h) - \beta_{n, k}(0)|1_{\mathcal{P}_0} \leq \sum_{i = 1}^{S_{0, h}}|\beta_{n, k}(\tau_i) - \beta_{n, k}(\tau_{i - 1})|1_{\mathcal{P}_0}.
\]
Using Lemma 2.2 from \cite{yogeshwaran2014random}, it then follows that
\begin{multline*}
|\beta_{n, k}(h) - \beta_{n, k}(0)|1_{\mathcal{P}_0} \leq \sum_{i = 1}^{S_{0,h}}|f_{n, k}(\tau_i) - f_{n, k}(\tau_{i - 1})|1_{\mathcal{P}_0} \\ + \sum_{i = 1}^{S_{0,h}}|f_{n, k + 1}(\tau_i) - f_{n, k + 1}(\tau_{i - 1})|1_{\mathcal{P}_0}.
\end{multline*}
However,  $|f_{n, k}(\tau_i) - f_{n, k}(\tau_{i - 1})| \leq \tbinom{n}{k + 1}$ and $|f_{n, k + 1}(\tau_i) - f_{n, k}(\tau_{i - 1})| \leq \tbinom{n}{k + 2}$. Hence,
\[
|\beta_{n, k}(h) - \beta_{n, k}(0)|1_{\mathcal{P}_0} \leq \left[\tbinom{n}{k + 1} + \tbinom{n}{k + 2}\right]S_{0,h} 1_{\mathcal{P}_0} \leq 2n^{k + 2} S_{0,h}.
\]
Similarly, if we let $S_{h, 2h}$ denote the total number of arrivals across edges in $(h,2h]$, then
\[
|\beta_{n, k}(2h) - \beta_{n, k}(h)|1_{\mathcal{P}_h} \leq 2n^{k + 2} S_{h,2h},
\]
where $\mathcal{P}_h$ denotes the event that no arrivals happened at time $h$. Since $1_{\mathcal{P}_0}$ and $1_{\mathcal{P}_h}$ are almost sure events, the above inequalities combined with \eqref{eqn:Term2} show that
\[
\text{Term}_2 \leq 16n^{4k + 8} \mathbb{E}[S^2_{0, h} S^2_{h, 2h} 1_{E^c}].
\]
Now using \eqref{eqn:EulerEqualBettiEvent}, note that
\[
1_{E^c} \leq 1_{\{(-1)^k\chi_{n}(0) \neq \beta_{n, k}(0)\}} + 1_{\{(-1)^k\chi_{n}(h) \neq \beta_{n, k}(h)\}} + 1_{\{(-1)^k\chi_{n}(2h) \neq \beta_{n, k}(2h)\}}.
\]
Consequently, we have
\begin{multline}
\label{eqn:Term2Bound}
\text{Term}_2 \leq 16n^{4k + 8} \bigg\{ \mathbb{E}\left[S^2_{0, h} S^2_{h, 2h} 1_{\{(-1)^k\chi_{n}(0) \neq \beta_{n, k}(0)\}}\right] \\ + \; \mathbb{E}\left[S^2_{0, h} S^2_{h, 2h}  1_{\{(-1)^k\chi_{n}(h) \neq \beta_{n, k}(h)\}}\right]  +  \mathbb{E}\left[S^2_{0, h} S^2_{h, 2h} 1_{\{(-1)^k\chi_{n}(2h) \neq \beta_{n, k}(2h)\}}\right]\bigg\}.
\end{multline}

However, for any $t \geq 0$, note that $1_{\{(-1)^k\chi_{n}(t) \neq \beta_{n, k}(t)\}}$ is a function of only $G(n, p, t)$ which in turn is a function of only $\{I_{i_{e}(t)}^e\}$, where
\[
i_{e}(t) := \min\{i: T_i^e \leq t < T_{i + 1}^e\}.
\]
Since for each $e$, the i.i.d.\ sequence $\{I^{e}_i\}$ and the sequence $\{T^e_i\}$ are independent, it is not difficult to see that  $\bigcup_e\{I_{i_{e}(t)}^e\}$ is independent of $\bigcup_e\{T^e_i\}$. So, $S_{0,h}, S_{h, 2h}$,(both of which depend only upon $\cup_e\{T^e_i\}$) and $1_{\{(-1)^k\chi_{n}(t) \neq \beta_{n, k}(t)\}}$ (which depends only upon $\cup_e\{I_{i_{e}(t)}^e\}$) are mutually independent for any $t \geq 0$. Since $S^2_{0,h}$ and $S^2_{h, 2h}$ are Poisson with parameter $\tbinom{n}{2} \lambda h$,
\[
\mathbb{E}[S^2_{0, h}] = \mathbb{E}[S^2_{h, 2h}] = \tbinom{n}{2} \lambda h + \tbinom{n}{2}^2 \lambda^2 h^2 \leq 2n^4 \lambda^2 h,
\]
where the last inequality follows since $0 \leq h \leq 1$. Consequently, we have
\begin{multline*}
\text{Term}_2 \leq 64 n^{4k + 16} \lambda^2 h^2 \bigg\{\Pr\{(-1)^k\chi_{n}(0) \neq \beta_{n, k}(0)\} \\  + \Pr\{(-1)^k\chi_{n}(h) \neq \beta_{n, k}(h)\} + \; \Pr\{(-1)^k\chi_{n}(2h) \neq \beta_{n, k}(2h)\}   \bigg\}.
\end{multline*}
However, from Theorem \ref{thm:VanishingOfAllButOneHomology},
\[
\Pr\{(-1)^k\chi_{n}(t) \neq \beta_{n}(t) \} = o(n^{-M}),
\]
for any  $M > 0$. Using this, it is not difficult to see that there exists $K^\prime_{\beta} > 0$ such that
\begin{equation}
\label{eqn:Term2UpperBound}
\frac{\text{Term}_2}{n^{4k} p^{4 \tbinom{k + 1}{2}  - 1}} \leq K_\beta^\prime h^2.
\end{equation}
Combining \eqref{eqn:BettiSplit}, \eqref{eqn:Term1Exp3}, and \eqref{eqn:Term2UpperBound}, the desired result follows. \end{proof}

Putting Lemma~\ref{lem:InitialDriftBettiNumbers} and Theorem~\ref{thm:TightnessBettiNumber} together shows that the sequence of processes $\{ \text{\BetaBar} : n \geq 1\}$ is tight. Combining this with Theorem~\ref{thm:AsymptoticMultivariateGaussianBettiNumber} completes the proof for Theorem~\ref{thm:ConvergenceBettiNumberOrnsteinUhlenbeckProcess} as desired. Note that along the way we have also proved that if $p = n^\alpha$ with $\alpha \in (-1/k, -1/(k + 1))$, then the sequences of processes $\{\text{\CliqueBar}: n \geq 1\}$ and $\{\text{\EulerBar}: n \geq 1\}$ converge in distribution to the stationary Ornstein-Uhlenbeck process.

\appendix

\section{The processes $\bar{\beta}_{n,k}$ are not Markovian}
\label{app:markov}

Although the dynamic \ER graph $\{G(n, p, t) : t \geq 0\}$ is a continuous time Markov chain, and the processes are
 $\{\beta_{n, k}(t) : t \geq 0\}$ are pointwise functions of them, they themselves are not Markovian. To prove this we need the following result from \cite[Theorem 4]{burke1958markovian}.

\begin{theorem}
%[Burke and Rosenblatt \cite{burke1958markovian}]
\label{thm:MarkovianFunction}
Let $\{X(t) : t\geq 0\}$ be a Markov chain on the state space $M=\{1,\dots,m\}$, with arbitrary initial distribution,   and stationary transition probability function $P(t) = (p_{ij}(t))$, continuous in $t$. Assume that $\lim_{t \rightarrow 0} P(t) = \mathbb{I}$.  Let $\psi$ be a function $M$ on  let $Y(t) = \psi(X(t))$. If the states if $Y$ are $y_1,\dots,y_r$, $r\leq m$, define $r$ disjoint subsets of $M$ by   $S_j = \{i\in M:\psi(i)=y_j\}$.  Then $Y$ is Markovian if, and only if, for each $j = 1, \ldots, r$, either one of the following conditions holds.
\begin{enumerate}
\item[(i)] $p_{i, S_{j}}(t) \equiv 0$ for all $i \notin S_j$.
\item[(ii)] $p_{i, S_{j}}(t) = C_{S_{j^\prime}, S_{j}}(t)$ for every $i \in S_{j^\prime}$ for $j^\prime = 1,\ldots, r$, where $C_{S_{j^\prime}, S_{j}}(t)$ is a constant that depends only on $S_{j^\prime}, S_{j}$, and $t$.
\end{enumerate}
(Note that (ii)$\implies$(i), and so (i) is irrelevant for the `only if' part of the theorem.)
\end{theorem}

An example which shows that the process $\{\beta_{n, k}(t) : t \geq 0\}$  is not Markov for finite $n$ is the following. Consider the dynamic \ER graph with $n = 4$, arbitrary $p \in (0,1)$, and arbitrary $\lambda > 0$. At any given time $t$, each of its $6$ edges, say $e_1, \ldots, e_6$, can be either in `on' or `off' state. Thus, $G(4, p, t)$ has $m = 64$ possible configurations. However the process $\{\beta_{4, 1}(t) : t \geq 0\}$ can only take $r = 2$ values, i.e.\ zero or one. This can be inferred from Figure~\ref{fig:SPA_BettiNotMarkov}, which gives the different edge configurations when $\beta_{4, 1}(t) = 1$, and the fact that if more than one of these configurations occur simultaneously then the resulting complex will have $\beta_{4,1}(t) =0$. Hence, using \eqref{eqn:EdgeTransition11} and \eqref{eqn:EdgeTransition00}, we have
\beqq
\Pr\{\beta_{4, 1}(t + s) = 1 | e_1(s) = \cdots = e_6(s) &=& \text{off}\} = 3 p^4 (1 - e^{-\lambda t})^4 ((1- p) + pe^{-\lambda t})^2,
\eeqq
while
\beqq
\Pr\{\beta_{4, 1}(t + s) = 1 | e_1(s) = \cdots = e_6(s) = \text{on}\} &=& 3 (p + (1 - p)e^{-\lambda t})^4 (1 - p)^2(1- e^{-\lambda t})^2.
\eeqq
Clearly, for a generic $p$ and $t$, the above two equations are unequal. On the other hand, $\beta_{4, 1}(s) = 0$ when either $e_1(s) = \cdots = e_6(s) = \text{off}$, or $e_1(s) = \cdots = e_6(s) = \text{on}$. These facts along with Theorem~\ref{thm:MarkovianFunction} show that the process $\{\beta_{4,1}(t) : t\geq 0\}$ is not Markovian.

\begin{figure}
\centering
\includegraphics[width = 0.6\textwidth,height=0.1\textwidth]{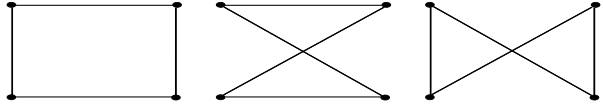}\\
\caption{ Configurations of $G(4, p, t)$ with $\beta_{4, 1}(t) = 1$. (No vertices at  intersections.)}
\label{fig:SPA_BettiNotMarkov}
\end{figure}

\section{Exact expression for $\mathbb{E}[g(h; \bar{A})]$}
\label{sec:app}

Consider the notations defined below Lemma~\ref{lem:TightnessEulerCharacteristic}. Clearly,
%\begin{equation}
%\label{eqn:gExpansion}
%\begin{array}{ccc}
%g(h; \bar{A}) & = & 1_{A_1}(2h)1_{A_2}(2h)1_{A_3}(h)1_{A_4}(h) + 1_{A_1}(h)1_{A_2}(h)1_{A_3}(0)1_{A_4}(0) \\
%& & +\;1_{A_1}(2h)1_{A_2}(2h)1_{A_3}(0)1_{A_4}(0) + 1_{A_1}(h)1_{A_2}(h)1_{A_3}(h)1_{A_4}(h) \\
%& & +\;1_{A_1}(2h)1_{A_2}(h)1_{A_3}(h)1_{A_4}(0) + 1_{A_1}(h)1_{A_2}(2h)1_{A_3}(h)1_{A_4}(0) \\
%& & +\;1_{A_1}(2h)1_{A_2}(h)1_{A_3}(0)1_{A_4}(h) + 1_{A_1}(h)1_{A_2}(2h)1_{A_3}(0)1_{A_4}(h) \\
%& & -\;1_{A_1}(2h)1_{A_2}(2h)1_{A_3}(h)1_{A_4}(0) - 1_{A_1}(2h)1_{A_2}(2h)1_{A_3}(0)1_{A_4}(h) \\
%& & -\; 1_{A_1}(h)1_{A_2}(h)1_{A_3}(0)1_{A_4}(h) - 1_{A_1}(h)1_{A_2}(h)1_{A_3}(h)1_{A_4}(0) \\
%& & -\;1_{A_1}(2h)1_{A_2}(h)1_{A_3}(h)1_{A_4}(h) - 1_{A_1}(h)1_{A_2}(2h)1_{A_3}(h)1_{A_4}(h) \\
%& & -\;1_{A_1}(2h)1_{A_2}(h)1_{A_3}(0)1_{A_4}(0) - 1_{A_1}(h)1_{A_2}(2h)1_{A_3}(0)1_{A_4}(0).
%\end{array}
%\end{equation}
\beq
\label{eqn:gExpansion}
g(h; \bar{A}) & = & 1_{A_1}(2h)1_{A_2}(2h)1_{A_3}(h)1_{A_4}(h) + 1_{A_1}(h)1_{A_2}(h)1_{A_3}(0)1_{A_4}(0) \\ \nonumber
&& +\;1_{A_1}(2h)1_{A_2}(2h)1_{A_3}(0)1_{A_4}(0) + 1_{A_1}(h)1_{A_2}(h)1_{A_3}(h)1_{A_4}(h) \\ \nonumber
&& +\;1_{A_1}(2h)1_{A_2}(h)1_{A_3}(h)1_{A_4}(0) + 1_{A_1}(h)1_{A_2}(2h)1_{A_3}(h)1_{A_4}(0) \\ \nonumber
 && +\;1_{A_1}(2h)1_{A_2}(h)1_{A_3}(0)1_{A_4}(h) + 1_{A_1}(h)1_{A_2}(2h)1_{A_3}(0)1_{A_4}(h) \\ \nonumber
 && -\;1_{A_1}(2h)1_{A_2}(2h)1_{A_3}(h)1_{A_4}(0) - 1_{A_1}(2h)1_{A_2}(2h)1_{A_3}(0)1_{A_4}(h) \\ \nonumber
 && -\; 1_{A_1}(h)1_{A_2}(h)1_{A_3}(0)1_{A_4}(h) - 1_{A_1}(h)1_{A_2}(h)1_{A_3}(h)1_{A_4}(0) \\ \nonumber
 && -\;1_{A_1}(2h)1_{A_2}(h)1_{A_3}(h)1_{A_4}(h) - 1_{A_1}(h)1_{A_2}(2h)1_{A_3}(h)1_{A_4}(h) \\ \nonumber
 && -\;1_{A_1}(2h)1_{A_2}(h)1_{A_3}(0)1_{A_4}(0) - 1_{A_1}(h)1_{A_2}(2h)1_{A_3}(0)1_{A_4}(0).
\eeq
Using \eqref{eqn:EdgeTransition11} and \eqref{eqn:EdgeProbability}, it is not difficult to see that if $\tau(h) := p + (1 - p) e^{-\lambda h}$, then $\mathbb{E}[g(h; \bar{A})] = p^{\text{pair}(\bar{A})} \Phi(h; \bar{A})$, where
\beq
\label{eqn:PhiVal}
&& \Phi(h; \bar{A}) =  \left[\tau(h)\right]^{\binom{a_{13}}{2} + \binom{a_{14}}{2} + \binom{a_{23}}{2} + \binom{a_{24}}{2} - \binom{a_{123}}{2} - \binom{a_{124}}{2} - \binom{a_{134}}{2} - \binom{a_{234}}{2} + \binom{a_{1234}}{2}}\\  \nonumber
&&+ \left[\tau(h)\right]^{\binom{a_{13}}{2} + \binom{a_{14}}{2} + \binom{a_{23}}{2} + \binom{a_{24}}{2} - \binom{a_{123}}{2} - \binom{a_{124}}{2} - \binom{a_{134}}{2} - \binom{a_{234}}{2} + \binom{a_{1234}}{2}}\\ \nonumber
&&+ \left[\tau(2h)\right]^{\binom{a_{13}}{2} + \binom{a_{14}}{2} + \binom{a_{23}}{2} + \binom{a_{24}}{2} - \binom{a_{123}}{2} - \binom{a_{124}}{2} - \binom{a_{134}}{2} - \binom{a_{234}}{2} + \binom{a_{1234}}{2}} + 1 \\ \nonumber
&&+\;\left[\tau(h)\right]^{\binom{a_{12}}{2} + \binom{a_{13}}{2} +  \binom{a_{24}}{2} + \binom{a_{34}}{2}- \binom{a_{234}}{2} - \binom{a_{123}}{2}}\left[\tau(2h)\right]^{\binom{a_{14}}{2} - \binom{a_{124}}{2} - \binom{a_{134}}{2} + \binom{a_{1234}}{2}} \\ \nonumber
&&+\;\left[\tau(h)\right]^{\binom{a_{12}}{2} + \binom{a_{23}}{2} +  \binom{a_{14}}{2} + \binom{a_{34}}{2}- \binom{a_{134}}{2} - \binom{a_{123}}{2}}\left[\tau(2h)\right]^{\binom{a_{24}}{2} - \binom{a_{124}}{2} - \binom{a_{234}}{2} + \binom{a_{1234}}{2}} \\ \nonumber
&&+\;\left[\tau(h)\right]^{\binom{a_{12}}{2} + \binom{a_{14}}{2} +  \binom{a_{23}}{2} + \binom{a_{34}}{2}- \binom{a_{234}}{2} - \binom{a_{124}}{2}}\left[\tau(2h)\right]^{\binom{a_{13}}{2} - \binom{a_{123}}{2} - \binom{a_{134}}{2} + \binom{a_{1234}}{2}} \\ \nonumber
&&+\;\left[\tau(h)\right]^{\binom{a_{12}}{2} + \binom{a_{24}}{2} +  \binom{a_{13}}{2} + \binom{a_{34}}{2}- \binom{a_{134}}{2} - \binom{a_{124}}{2}}\left[\tau(2h)\right]^{\binom{a_{23}}{2} - \binom{a_{123}}{2} - \binom{a_{234}}{2} + \binom{a_{1234}}{2}} \\ \nonumber
&&-\;\left[\tau(h)\right]^{\binom{a_{13}}{2} + \binom{a_{23}}{2} + \binom{a_{34}}{2} - \binom{a_{123}}{2}} \left[\tau(2h)\right]^{\binom{a_{14}}{2} + \binom{a_{24}}{2}- \binom{a_{124}}{2} - \binom{a_{134}}{2} - \binom{a_{234}}{2} + \binom{a_{1234}}{2}} \\ \nonumber
&&-\;\left[\tau(h)\right]^{\binom{a_{14}}{2} + \binom{a_{24}}{2} + \binom{a_{34}}{2} - \binom{a_{124}}{2}} \left[\tau(2h)\right]^{\binom{a_{13}}{2} + \binom{a_{23}}{2}- \binom{a_{123}}{2} - \binom{a_{134}}{2} - \binom{a_{234}}{2} + \binom{a_{1234}}{2}} \\ \nonumber
&&-\;\left[\tau(h)\right]^{\binom{a_{14}}{2} + \binom{a_{24}}{2} + \binom{a_{34}}{2} - \binom{a_{124}}{2} - \binom{a_{134}}{2} - \binom{a_{234}}{2} + \binom{a_{1234}}{2}} \\ \nonumber
&&-\;\left[\tau(h)\right]^{\binom{a_{13}}{2} + \binom{a_{23}}{2} + \binom{a_{34}}{2} - \binom{a_{123}}{2} - \binom{a_{134}}{2} - \binom{a_{234}}{2} + \binom{a_{1234}}{2}} \\ \nonumber
&&-\;\left[\tau(h)\right]^{\binom{a_{12}}{2} + \binom{a_{13}}{2} + \binom{a_{14}}{2}- \binom{a_{123}}{2} - \binom{a_{124}}{2} - \binom{a_{134}}{2} + \binom{a_{1234}}{2}} \\ \nonumber
&&-\;\left[\tau(h)\right]^{\binom{a_{12}}{2} + \binom{a_{23}}{2} + \binom{a_{24}}{2}- \binom{a_{123}}{2} - \binom{a_{124}}{2} - \binom{a_{234}}{2} + \binom{a_{1234}}{2}} \\ \nonumber
&&-\;\left[\tau(h)\right]^{\binom{a_{12}}{2} + \binom{a_{23}}{2} + \binom{a_{24}}{2} - \binom{a_{234}}{2}}\left[\tau(2h)\right]^{\binom{a_{13}}{2} + \binom{a_{14}}{2} - \binom{a_{123}}{2} - \binom{a_{124}}{2} - \binom{a_{134}}{2} + \binom{a_{1234}}{2}} \\ \nonumber
&&-\;\left[\tau(h)\right]^{\binom{a_{12}}{2} + \binom{a_{13}}{2} + \binom{a_{14}}{2} - \binom{a_{134}}{2}}\left[\tau(2h)\right]^{\binom{a_{23}}{2} + \binom{a_{24}}{2} - \binom{a_{123}}{2} - \binom{a_{124}}{2} - \binom{a_{234}}{2} + \binom{a_{1234}}{2}}.
\eeq

\acks
The authors would like to  thank an anonymous Referee, Vivek Borkar, and Sandeep Juneja for  useful comments and suggestions. The research of GT was supported in part by an IBM PhD fellowship and TOPOSYS, FD7-ICT-318493 STREP. DY was supported by a DST-INSPIRE Faculty Award and  TOPOSYS. RJA was partially supported under URSAT, ERC Advanced Grant 320422.

\bibliographystyle{apt}
\bibliography{AAP_Bib}

\end{document}